 \documentclass[final,11pt, times, linenumber, nopreprintline]{elsarticle}

\usepackage[margin=2.5cm]{geometry}

\usepackage{framed,multirow}

\usepackage{csquotes}
\usepackage{amsmath}
\allowdisplaybreaks
\usepackage{amssymb}
\usepackage{commath}
\usepackage{mathtools}
\usepackage{bbm}
\usepackage{amsthm}
 \newtheorem{example}{Example}
\usepackage{relsize}
\usepackage{adjustbox}
\usepackage{algorithm}
\usepackage[noend]{algpseudocode}
\usepackage{booktabs}
\usepackage{tikz}
\usepackage{bm}
\usepackage{soul}
\usepackage{lineno}

\usepackage{url}
\usepackage{hyperref}

\DeclareMathOperator\supp{supp}

\newcommand{\N}{\mathbb{N}}
\newcommand{\R}{\mathbb{R}}

\makeatletter
\newcommand*\dashline{\rotatebox[origin=c]{90}{$\dabar@\dabar@\dabar@$}}
\makeatother
\def\TTV{\text{TTV}}

\DeclarePairedDelimiterX\newset[1]\lbrace\rbrace{\setaux #1||\endsetaux}
\def\setaux#1|#2|#3\endsetaux{\if\relax\detokenize{#2}\relax #1 \else #1 \;\delimsize\vert\; #2 \fi}

\let\epsilon\varepsilon
\let\phi\varphi
\let\rho\varrho

\newcommand{\dt}{\Delta t}

\newcommand{\bA}{\bm A}
\newcommand{\dx}{\operatorname{d}\!x}
\renewcommand{\dt}{\operatorname{d}\!t}

\newcommand{\bu}{ u}
\newcommand{\bU}{\bm U}

\newcommand{\bg}{ g}

\newcommand{\bF}{\bar{F}}
\renewcommand{\bf}{ f}

\newtheorem{thm}{Theorem}
\newtheorem{lem}[thm]{Lemma}

\newdefinition{rmk}{Remark}
\newproof{pf}{Proof}
\newproof{pot}{Proof of Theorem \ref{thm2}}
\newdefinition{dfn}{Definition}

\begin{document}


\begin{frontmatter}
	
\title{The Lax--Wendroff theorem for Patankar-type methods applied to hyperbolic conservation laws\footnote{%
		© 2025 The Author(s). Published in \textit{Computers \& Fluids} (Elsevier) under the Creative Commons Attribution 4.0 International License (CC BY 4.0). 
		DOI: \href{https://doi.org/10.1016/j.compfluid.2025.106885}{10.1016/j.compfluid.2025.106885}. 
		License: \href{https://creativecommons.org/licenses/by/4.0/}{creativecommons.org/licenses/by/4.0/}.
}}

\author[1]{Janina {Bender}\corref{cor1}}
\cortext[cor1]{Corresponding author:}
\ead{jmbender@mathematik.uni-kassel.de}
\author[1]{Thomas {Izgin}}
\author[3]{Philipp {Öffner}\corref{cor3}}
\author[4]{Davide {Torlo}}

\affiliation[1]{organization={Institute of Mathematics, University of Kassel},
addressline={Heinrich-Plett-Str. 40}, 
postcode={34132},
city={Kassel},
country={Germany}}
\affiliation[3]{organization={Institute of Mathematics, Clausthal University of Technology}, 
city={Clausthal-Zellerfeld}, 
country={Germany}}
\affiliation[4]{organization={Dipartimento di Matematica Guido Castelnuovo, Università di Roma La Sapienza}, 
addressline={piazzale Aldo Moro, 5}, 
city={Roma}, 
postcode={00185}, 
country={Italy}}


\begin{abstract}

For hyperbolic conservation laws, the famous Lax--Wendroff theorem delivers sufficient conditions for the limit of a convergent numerical method to be a weak (entropy) solution.
This theorem is a fundamental result, and many investigations have been done to verify its validity for finite difference, finite volume, and finite element schemes, using either explicit or implicit linear time-integration methods.

Recently, the use of modified Patankar (MP) schemes as time-integration methods for the discretization of hyperbolic conservation laws has gained increasing interest. These schemes are unconditionally conservative and positivity-preserving and only require the solution of a linear system. However, MP schemes are by construction nonlinear, which is why the theoretical investigation of these schemes is more involved. 
We prove an extension of the Lax--Wendroff theorem for the class of MP methods. This is the first extension of the Lax--Wendroff theorem to nonlinear time integration methods with just an additional hypothesis on the total time variation boundedness of the numerical solutions. We provide some numerical simulations that validate the theoretical observations.


%
\end{abstract}
\end{frontmatter}


\section{Introduction}


In the context of hyperbolic conservation laws, the development of high-order, structure-preserving methods has attracted a broad attention. These methods are essential to ensure physically reliable solutions satisfying suitable constraints. One such constraint is the positivity of certain quantities, such as density in the Euler equations of gas dynamics or the water height in the context of shallow water flows where positivity preservation is usually obtained within a numerical method by a positivity-preserving reconstruction/limiting technique and an additional time step constraint,  see for instance \cite{zbMATH07745801, RB2009, ABBKP2004}. 
To overcome the limitation of time step constraints without solving a completely nonlinear system and to guarantee the preservation of positivity, modified Patankar (MP) schemes have recently been applied, see e.g. \cite{ciallella2022arbitrary,huang2019third,ciallella2025high}. These schemes, designed for production-destruction systems (PDS), are nonlinear, they preserve positivity unconditionally, and are conservative \cite{burchard_high-order_2003, offner2020arbitrary}. 
Moreover, due to their specific structure, only linear equations need to be solved. Applying them as time-integration schemes for hyperbolic problems has yielded promising results, as outlined in recent works \cite{MO2014,Ortleb2014, OrtlebH2017}.
For instance in \cite{ciallella2022arbitrary,ciallella2025high}, the MP Deferred Correction (MPDeC) method from \cite{offner2020arbitrary} is used to ensure the positivity of the water height in the shallow water equations discretized with finite volumes in space.
In \cite{MO2014,Ortleb2014}, it is also ensured with a Patankar-type method, where a discontinuous Galerkin (DG) method is used for space discretization. 
However, the theoretical background of these schemes has not yet been sufficiently investigated, especially concerning their integration into the existing convergence theory for numerical methods of hyperbolic systems. 
A cornerstone in the numerical treatment of hyperbolic conservation laws within this framework is the well-known Lax--Wendroff theorem \cite{lax_systems_1960}. 
This theorem establishes sufficient conditions for the convergence of numerical approximations to weak (entropy) solutions of the underlying nonlinear hyperbolic conservation law. 
Much of the literature, see e.\,g.\ \cite{abgrall2023, eymard2022finite, kuzmin2022limiter,shi_2018_local, GRP_2024, arXiv:2308.14872,Weak_FV_2019,SS2018,zbMATH07745801}, has been focusing the on spatial discretization (even with unstructured grids and finite element (FE) methods), mostly using 
linear explicit time integration methods.
One exception is given in \cite{birken_conservation_2022, LB24}, where a Lax--Wendroff theorem was proven for implicit Runge--Kutta methods using explicit pseudo-time Runge--Kutta schemes for approximating the solution to the occurring (non-)linear system. 
However, here only finitely many pseudo-time steps are allowed to approximate the update of the difference scheme. 
Given a fixed spatial discretization of $N$ points, this is not a problem for linear implicit systems, as suitable Krylov subspace methods converge to the exact solution in $N$ iterations.
Problems arise when we are interested in the convergence towards the exact solution for $N\to \infty$.
In this case, the previous result does not directly apply, leaving some gaps in the proof of the Lax-Wendroff theorem for the implicit setting, either linear or nonlinear.
Nevertheless, a first step towards a deeper understanding is given in \cite{LB23_Krylov_consistency}. 

In this paper, we propose a different ansatz introducing the notion of total time variation and we shift the focus to \textit{nonlinear} time-integration schemes extending the convergence theory of Lax and Wendroff to Patankar-type methods. Altogether, the extension of the Lax--Wendroff theorem additionally assumes that the total variation in the time variable (see Definition~\ref{def:TTV}) is uniformly bounded.
This extension allows us, for the first time, to establish a theory of their convergence in the context of partial differential equations (PDEs). The theoretical results are validated through numerical tests for different conservation laws.

The rest of the paper is organized as follows: In Section~\ref{Se_preliminaries}, we introduce the concept of consistency and convergence for difference schemes for conservation laws and repeat the basic formulation of the Lax--Wendroff theorem which can be found nowadays in many excellent lecture books like \cite{leveque_numerical_1992, kroner_numerical_1997,godlewski_numerical_1996}.
Afterwards, in Section~\ref{se_Patankar}, MP schemes will be introduced for production-destruction systems in context of ordinary differential equations (ODE). Here, the recent interpretation as  non-standard additive Runge--Kutta (NSARK) methods \cite{NSARK} will also be used.
Then, we explain how to employ the MP ideas within difference schemes (finite volume (FV), finite differences, for example with Weighted Essentially Non-Oscillatory (WENO) reconstruction) to guarantee positivity preservation. 
In Section~\ref{se_Main}, we demonstrate the extension of the classical Lax--Wendroff Theorem for a large class of NSARK methods, including MP difference schemes. 
Here, we focus on the first order version (Theorem~\ref{th_Main_MP}) but we give also a sketch of the proof for the extension to higher-order. 
In Section~\ref{se_numerics}, we validate our theoretical findings by some numerical experiments.
In  particular, we demonstrate that we are unconditionally positivity-preserving and we converge to a weak solution. 
In Section~\ref{se_conclusion}, we summarize the work, we derive conclusions and give perspectives for future works. 
Finally, in \ref{sec:TVD}, we demonstrate how the total variation, which is an essential aspect of the Lax--Wendroff theorem, can in some specific frameworks be proven to be diminishing for the studied MP schemes. 

\section{Consistency and convergence for difference schemes for conservation laws }\label{Se_preliminaries}
To transfer the convergence theory of the Lax--Wendroff theorem to Patankar-type methods, we first summarize the theoretical background of the classical theory. We consider a one-dimensional, scalar Cauchy problem
\begin{equation}
	\partial_t \bu(x,t) + \partial_x f(u(x,t)) = 0, \quad \bu(x,0)=\psi(x)\label{eq:InitProb},
\end{equation}
with $\bu\colon \R\times \R^+_0\to \R$ and $f,\psi\colon \R\to \R$, where $f$ denotes the physical flux function.

We examine difference schemes of the form
\begin{equation}
	U_i^{n+1} = U_i^n - \frac{\Delta t}{\Delta x}\left(\bg_{i+\frac{1}{2}}^n - \bg_{i-\frac{1}{2}}^n \right)
	\label{eq:diffscheme}
\end{equation}
to find approximations $U_i^n \in \R$ to the solution $\bu(x_i,t_n)$ at $N$ discrete grid points $\lbrace x_i\rbrace_{i=1}^N$ and at time points $\lbrace t_n \rbrace_{n=0}^{N_t}$. 
Here, the time step of a given grid is denoted by $\Delta t = t_{n+1}-t_n$ for all $n$ and the spatial mesh size by $\Delta x=x_{i+1}-x_{i}$ for all $i$. Furthermore, we have
\begin{equation}
	\bg_{i+ \frac{1}{2}}^n = \bg(U^n_{i-p}, U^n_{i-p+1}, \dots, U^n_{i+q}),
	\label{eq:numflux}
\end{equation}
where $\bg$ is a continuous function, depending on $p + q + 1$ arguments, called the \textit{numerical flux}. The following results and definitions hold for both explicit and implicit numerical fluxes \cite[p. 44]{kroner_numerical_1997}, but for simplicity, we only consider explicit schemes. First, we note the main properties of these fluxes.
 
\begin{dfn}[Consistency]
	\label{def:cons}
The numerical flux function $\bg$ from \eqref{eq:numflux} is \textit{consistent} to the physical flux function $\bf$, if
\begin{enumerate}[i.]
\item  $g(\bar{\bu},\dots,\bar{\bu}) = f(\bar{\bu})$ for any $\bar{u}$ and
\item $\bg$ is locally Lipschitz continuous in every argument, that is if it exists a $K \in \R$ such that

$|\bg(U^n_{i-p}, U^n_{i-p+1}, \dots, U^n_{i+q})-\bf(\bar{\bu})| \leq K \max_{-p \leq j \leq q} |U^n_{i+j} - \bar{\bu}|$ holds
for $U^n_{i+j}$ sufficiently close to $\bar{\bu}$.
\end{enumerate}
\end{dfn}
If $g$ is consistent, then we also say that the difference scheme \eqref{eq:diffscheme} is consistent with the PDE \eqref{eq:InitProb}.

To ensure that the limit of a converging difference scheme is a weak solution of the system, \textit{Lax} and \textit{Wendroff} \cite{lax_systems_1960} proved the following theorem. Here, we present a formulation with slightly different requirements.
%
%
%
%
\begin{thm}[Lax--Wendroff \cite{leveque_numerical_1992}]
	Consider a sequence of step sizes $\Delta x_{\{l\}} > 0$ with  $ \lim_{l{ \rightarrow \infty}} \Delta x_{\{l\}} = 0$. Let $\left( U^{\{l\}} \right)_l $ be a sequence of piece-wise constant functions constructed from numerical solutions obtained by a difference scheme consistent with \eqref{eq:InitProb} with time step size $\Delta t_{\{l\}} = \lambda \Delta x_{\{l\}}, \lambda \equiv \mathrm{const}.$ If  for any domain $[a,b]\times [0,T]$  the total variation $TV \left(U^{\{l\}}(\cdot, t)\right)$ is uniformly bounded for all $l$ and $t \in [0,T]$ and if $U^{\{l\}}\xrightarrow{l \rightarrow \infty} u$ in $L^1_{\mathrm{loc}}(\R\times [0,T))$, then $u$ is a weak solution of the conservation law \eqref{eq:InitProb}.
	\label{thm:LW}
\end{thm}
\begin{rmk} The Lax--Wendroff theorem \ref{thm:LW} guarantees that the limit of a convergent conservative numerical scheme is a weak solution of the conservation law. Furthermore, if the approximated solutions also satisfy a discrete form of the entropy inequality, then the limit corresponds to the weak entropy solution. Therefore, the theorem remains applicable in identifying the entropy solution, provided the entropy condition is met. Therefore, applying entropy stable schemes which fulfill the Lax--Wendroff theorem yield weak entropy solutions \cite{Tadmor2003,CF2013,fisher20213,godlewski_numerical_1996}. 
\end{rmk}
The main idea of this paper is to extend Theorem \ref{thm:LW} if MP schemes (therefore nonlinear time-integration) schemes are used inside the discretization. We shortly introduce now the concept of MP schemes in the following section.

\section{Patankar-type schemes}\label{se_Patankar}
\subsection{Modified Patankar  for Production-Destruction Systems}
Patankar-type schemes are introduced to solve systems of ODE problems
\begin{equation}
	 \partial_t\mathbf{u}(t) = \Phi(\mathbf{u}(t)), \quad \mathbf{u}(t_0) = \mathbf{u}^0,
\end{equation}
with $\mathbf{u}: \R_0^+ \rightarrow \R^N$ and $\Phi: \R^N \rightarrow \R^N$.
Similar to PDE problems, we denote the $i$-th component of an approximate solution obtained by a numerical scheme at time step $t_n$ as $U^n_i$, while $\bU^n$ denotes the whole solution vector.
As mentioned in the introduction, MP schemes are constructed to preserve positivity and conservativity, properties that are defined as follows.
\begin{dfn}
	An ODE scheme is called \textit{unconditionally positive} if $\bU^n>0$ implies $\bU^{n+1}>0$ for all $n\in\N_0$ and $\Delta t > 0$. Furthermore it is \textit{unconditionally conservative} if $\sum_{i=1}^{N} U_i^{n+1}=\sum_{i=1}^{N} U_i^n $ for all $n \in \N_0$ and $\Delta t > 0$.
\end{dfn}
Throughout the paper, inequalities involving vectors, such as $\bU^n>0$, are meant componentwise for all components.
Consider an ODE written as a PDS
\begin{equation}
\partial_t u_i(t)=\sum_{j=1}^N p_{ij}(\mathbf u(t))-\sum_{j=1}^N d_{ij}(\mathbf u(t)), \quad i=1,\ldots,N 
\label{eq:PDS}
\end{equation}
with $p_{ij},d_{ij}$ non-negative production and destruction terms for $\mathbf{u}\geq 0$. In the past, such PDSs were commonly used to describe \textit{inter alia} biological reactions as in \cite{burchard_high-order_2003}. In later works, PDSs were also used to describe the semi-discretization of hyperbolic balance laws as in \cite{offner2020arbitrary, MO2014}. For both contexts it is important to have conservative and unconditionally positive schemes. Therefore, it is necessary to denote these properties also for PDSs.
\begin{dfn}
	A PDS \eqref{eq:PDS} is said to be \textit{non-negative} if, for an initial value $u_i(0)\geq 0$, $u_i(t)\geq 0$ follows for all $t > 0$ and $i = 1, \dots, N$. If the strict inequalities hold, the PDS is called \textit{positive}.
	\end{dfn}

\begin{dfn}
	The PDS \eqref{eq:PDS} is called \textit{conservative} if $p_{ij}(\mathbf u)=d_{ji}(\mathbf u)$ for all $i,j=1, \dots, N$ and \textit{fully conservative} if additionally $p_{ii}(\mathbf u) = d_{ii}(\mathbf u) = 0$ for all $i=1,\dots,N$.
\end{dfn}
We remark that a conservative and a fully conservative PDS are analytically the same model, but treating the production and destruction matrices numerically, it is safer to choose the fully conservative one. Hereafter, all methods are supposed to be applied to a fully conservative and positive PDS.
Based on such a conservative and positive PDS we can apply an explicit Runge--Kutta (RK) method and use the modified Patankar-trick to obtain a modified Patankar--Runge--Kutta (MPRK) scheme. 
	\begin{dfn}
	Given an explicit $s$-stage RK scheme with non-negative Butcher tableau $\arraycolsep=1.4pt\def\arraystretch{1.15}\begin{array}{c|c}
	\mathbf{c} & \mathbf{A} \\
	\hline
			   & \mathbf{b}^T
	\end{array}$, i.e., $\mathbf A \geq 0$ and $\mathbf b \geq 0$, applied to a fully conservative PDS \eqref{eq:PDS},
	the method
	\begin{align}
		U_i^{(k)} &= U_i^n + \Delta t\sum_{\nu=1}^{k-1} a_{k\nu}
		\sum_{j=1}^N\biggl( p_{ij}(\mathbf{U}^{(\nu)})\frac{U_j^{(k)}}{\pi^{(k)}_j}-d_{ij}(\mathbf{U}^{(\nu)})\frac{U_i^{(k)}}{\pi^{(k)}_i}\biggr), \quad k=1,\dotsc,s\label{eq:MPRK_stages}\\	
        U_i^{n+1}&=U_i^n+\Delta t \sum_{k=1}^s b_k
		\sum_{j=1}^N\biggl( p_{ij}(\mathbf{U}^{(k)})\frac{U_j^{n+1}}{\sigma_j}-d_{ij}(\mathbf{U}^{(k)})\frac{U_i^{n+1}}{\sigma_i}\biggr)\label{eq:MPRK_final_update}
	\end{align}
	is called \emph{MPRK scheme} if the Patankar weight denominators (PWDs) $\pi_i^{(k)}$ and $\sigma_i$ are unconditionally positive, $\pi_i^{(k)}=\pi_i^{(k)}(\bU^n,\bU^{(1)}, \dots, \bU^{k-1})$ is independent of $\mathbf U^{(k)}$, and $\sigma_i=\sigma_i(\bU^n,\bU^{(1)}, \dots,\bU^{(s)})$ is independent of $\mathbf U^{n+1}$.
	\label{def:MPRK}
\end{dfn}
	The huge advantage of such methods is that, given PWDs, MPRK schemes are unconditionally positive and conservative scheme, see for instance \cite{kopecz_order_2018}, while they merely require the solution of $s$ linear systems
\begin{align*}
	\bU^{(1)}&=\bU^n,\\
	\mathbf{M}^{(k)}\; \bU^{(k)}&=\bU^n,\quad k=2,\dots,s,\\
	\mathbf{M} \; \bU^{n+1}&=\bU^n,
\end{align*}
all of which are well-posed having a unique and positive solution. We note that MPRK schemes can also be devised for non-conservative PDS, see e.g. \cite{TOR2022,IzginThesis}, which is relevant when applying these schemes in the context of PDEs with non-periodic boundary conditions.

\begin{example}[Modified Patankar Euler]
As an example, we illustrate here the first order modified Patankar Euler (MPE) time discretization of a conservative PDS.
In this case, there is only the final update \eqref{eq:MPRK_final_update} where $\sigma_i=U^n_i$ for all $i=1,\dots,N$. The scheme reads
\begin{equation*}
        U_i^{n+1}=U_i^n+\Delta t \sum_{j=1}^N\biggl( p_{ij}(\mathbf{U}^{n})\frac{U_j^{n+1}}{U_j^{n}}-d_{ij}(\mathbf{U}^{n})\frac{U_i^{n+1}}{U_i^{n}}\biggr).\label{eq:MPE_final_update}
\end{equation*}
Clearly, each time step includes a linear system to be solved, with a matrix defined by
\begin{equation*}\label{eq:matrix_MPE}
    M_{ij}= \begin{cases}
        -\Delta t \frac{p_{ij}(\mathbf{U}^n)}{U_j^n} & \text{if } i\neq j,\\
        1+\sum_{k\neq i} \Delta t \frac{d_{ik}(\mathbf{U}^n)}{U_k^n} & \text{if } i= j.
    \end{cases}
\end{equation*}
This scheme will be the prototype for high order methods and we will use it to define the successive examples.
\end{example}

After the first introduction of MPRK methods based on explicit first- and second-order RK schemes in \cite{burchard_high-order_2003}, the Patankar-trick has since been successfully extended to a variety of RK frameworks, including state-of-the-art strong-stability-preserving Runge–Kutta (SSPRK) methods \cite{shu_efficient_1988} resulting in MPSSPRK schemes \cite{huang_positivity-preserving_2019}, and arbitrary high-order Deferred Correction (DeC) schemes \cite{dutt_spectral_2000} yielding MPDeC methods \cite{ciallella2022arbitrary}. It has also been applied to multistep methods in a recent work \cite{IMPV2025}. 
 Further positivity-preserving schemes are Geometric Conservative (GeCo) \cite{MCD2020} and generalized BBKS (gBBKS) \cite{AKM2020}. To give the main ideas for the extension of the Lax--Wendroff theorem to all these schemes and many more, we use the environment of NSARK methods as introduced in \cite{IzginThesis} to obtain results that are as universal as possible.

Therefore, first recall the definition of additive Runge--Kutta (ARK) methods. These are constructed based on an additive splitting
$$\frac{\mathrm{d}}{\dt} \mathbf{u}(t)  = \sum_{\substack{\nu=1}}^N  \mathbf{F}^{[\nu]}(\mathbf{u}(t)), \,\,\,  \mathbf{u}(0)  = \mathbf{u}^0\in \R^d$$
 of the right hand side of the ODE.
Thus, several schemes given by an extended Butcher tableau
		\[\arraycolsep=1.4pt\def\arraystretch{1.5}
		\begin{array}{c|c|c|c|c}
			\mathbf c &	\bA^{[1]}        &    \bA^{[2]}  & \cdots & \bA^{[N]} \\ \hline
			&	\mathbf b^{[1]}         &    \mathbf b^{[2]}  & \cdots & \mathbf b^{[N]} 
		\end{array},
		\]
		can be applied to the different summands leading to a scheme of the form
	\begin{align*}
		\bU^{(k)}& = \bU^n + \Delta t \sum_{r=1}^s  \sum_{\substack{\nu=1}}^N a^{[\nu]}_{kr} \mathbf{F}^{[\nu]}(\bU^{(r)}), \quad k=1,\dotsc,s,\\
		\bU^{n+1} & = \bU^n + \Delta t \sum_{r=1}^s \sum_{\substack{\nu=1}}^N b^{[\nu]}_r \mathbf{F}^{[\nu]}(\bU^{(r)}).
\end{align*}
Notice that these Runge--Kutta coefficients are fixed for every summand.
Now, the idea of non-standard additive Runge--Kutta (NSARK) methods is to allow a dependency on both solution and stepsize in the coefficients, which read as
	\[\arraycolsep=1.4pt\def\arraystretch{1.5}
	\begin{array}{c|c|c|c|c}
		\mathbf c &	\bA^{[1]}(\hat{\bU}^n,t_n,\Delta t)         &    \bA^{[2]}(\hat{\bU}^n,t_n,\Delta t)  & \cdots & \bA^{[N]}(\hat{\bU}^n,t_n,\Delta t) \\ \hline
		&	\mathbf b^{[1]}(\hat{\bU}^n,t_n,\Delta t)        &    \mathbf b^{[2]}(\hat{\bU}^n,t_n,\Delta t)  & \cdots & \mathbf b^{[N]}(\hat{\bU}^n,t_n,\Delta t) 
	\end{array},
	\]
where $\hat{\bU}^n=\left(\bU^n \dashline\bU^{(1)}\dashline\dotsc\dashline\bU^{(s)}\dashline \bU^{n+1}\right)$,	resulting in
\begin{align*}
	\bU^{(k)} & = \bU^n + \Delta t \sum_{r=1}^s  \sum_{\substack{\nu=1}}^N a^{[\nu]}_{kr}(\hat{\bU}^n,t_n,\Delta t)  \mathbf{F}^{[\nu]}(\bU^{(r)}), \quad k=1,\dotsc,s,\\
	\bU^{n+1} & = \bU^n + \Delta t \sum_{r=1}^s \sum_{\substack{\nu=1}}^N b^{[\nu]}_r(\hat{\bU}^n,t_n,\Delta t)  \mathbf{F}^{[\nu]}(\bU^{(r)}).
	\end{align*}
As introduced in equation \eqref{eq:PDS}, we focus on ODEs which can be split into a PDS. So we can understand such a right-hand side (RHS) as
\[F^{[\nu]}_i(\bU(t))=\begin{cases}p_{i\nu}(\bU(t)), &\nu\neq i,\\-\sum_{j=1}^Nd_{ij}(\bU(t)), &\nu=i \end{cases}\]
and interpret a Patankar-type method as NSARK method with 
\begin{equation}\label{eq:NSW}
    \begin{aligned}
        a_{kr}^{[\nu]}(\hat{\bU}^n,t_n,\Delta t)&=a_{kr}\gamma_{\nu}^{(k)}(\hat{\bU}^n,t_n,\Delta t),\\
        b_{r}^{[\nu]}(\hat{\bU}^n,t_n,\Delta t)&=b_{r}\delta_{\nu}(\hat{\bU}^n,t_n,\Delta t),
    \end{aligned}
\end{equation} 
where the Patankar-trick is incorporated into the \textit{non-standard weights (NSWs)} $\gamma_{\nu}^{(k)}, \delta_\nu$.
In particular, we obtain the MPRK methods introduced in Definition \ref{def:MPRK} by choosing the NSWs
\[\gamma_{\nu}^{(k)}=\frac{U_\nu^{(k)}}{\pi_\nu^{(k)}}, \  \delta_{\nu}=\frac{U_\nu^{n+1}}{\sigma_\nu}.\]
Note that we use \enquote{$()$} in the upper indices whenever we refer to a stage while \enquote{$[]$} is used to refer to a different RK scheme used within the ARK framework.

\subsection{Modified Patankar for difference schemes}

The key properties of Patankar-type methods, being unconditionally positivity-preserving and conservative, can be a huge advantage when used as the time integration scheme of the semidiscretization of a hyperbolic PDE. 
For instance, they can guarantee to produce nonnegative solutions when approaching zero. 
This is in contrast with almost all classical time discretizations where this bound can either be obtained with constraints on the time-step or it cannot be obtained at all. 
Negative solutions in such contexts can lead not only to unphysical solutions, but also to NaN in the code and stop the execution.
The only major exception is the implicit Euler method, which is unfortunately only first order accurate and very dissipative for hyperbolic conservation laws.

To apply Patankar-type time integration methods to semidiscrete hyperbolic PDEs, it is necessary to find and adequate splitting of the obtained ODE into a PDS. After that, the production and destruction terms are properly weighted. Note that such a splitting is not unique, see Remark~\ref{rem:other_splittings}. 

The transfer to Patankar-type difference schemes is done along the difference scheme formulation as originally introduced in \cite{lax_systems_1960}. That is, a flux formulation of the RHS treated by an explicit Euler scheme in time, which in our context is associated with the MPE scheme. Thus, the modification of the numerical fluxes is the key point of this extension.
Therefore, we first introduce an explicit conservative PDS based on a classical numerical flux $g$ as in \eqref{eq:diffscheme} and add the corresponding weights. These are kept in a general notation to clarify the later extension to general Patankar-type methods with more stages.

Starting with the PDS, we interpret the fluxes at the cell boundaries as production or destruction terms according to their signs, i.e., 
\begin{equation}
\begin{aligned}
	p_{i,i+1}(\bU^{n}) &= d_{i+1,i}(\bU^{n}) := -\min\bigl\{g_{i+\frac{1}{2}}^n,0\bigr\}, \\
	d_{i,i+1}(\bU^{n}) &= p_{i+1,i}(\bU^{n}) := \max\bigl\{g_{i+\frac{1}{2}}^n,0\bigr\},
	\label{eq:prod_dest}
    \end{aligned}
\end{equation}
where $\bU^n$ denotes the vector containing $U_i^n$ for $i=1, \dots, N$. 
With that, we can set $p_{i,j}, d_{i,j}$ to zero for $j \notin \{i-1,i+1\}\subseteq\{1,\dotsc,N\}$.
Now, we consider a one-stage Patankar-type method interpreted as a non-standard additive Runge--Kutta (NSARK) method with non-standard weights (NSWs) $\delta_\nu = \delta_\nu(\hat\bU^n,t_n, \Delta t)$ \cite{IzginThesis}. We write it in terms of this PDS in difference form:
\begin{equation}
	\begin{split}
		U_i^{n+1} &= U_i^n + \Delta t \sum_{j = 1}^N \frac{1}{\Delta x}\biggl(  p_{i,j}(\bU^n)\delta_j - d_{i,j}(\bU^n)\delta_i \biggr) \\
		&= U_i^n + \frac{\Delta t}{\Delta x} \biggl(p_{i,i+1}(\bU^{n})\delta_{i+1} - d_{i,i+1}(\bU^{n})\delta_i + p_{i,i-1}(\bU^{n}) \delta_{i-1} - d_{i,i-1}(\bU^{n})\delta_i \biggr),
	\end{split}
	\label{eq:MPRK_difference}
\end{equation}
where we use the conservation property of the PDS, i.e., $p_{i,j}=d_{j,i}$, leading to the difference scheme formulation 
\begin{equation*}
    U^{n+1}_i = U^n_i - \frac{\Delta t}{\Delta x} \left( F^{\delta}_{g_{i+\frac{1}{2}}}(\bU^n) -F^{\delta}_{g_{i-\frac{1}{2}}}(\bU^n) \right)
\end{equation*}
with the new numerical flux
\begin{equation}
	F^{\delta}_{g_{i+\frac{1}{2}}}(\bU^n):= d_{i,i+1}(\bU^n) \delta_i - p_{i, i+1}(\bU^n) \delta_{i+1} = \max \lbrace g^{n}_{i+\frac12},0 \rbrace \delta_i + \min \lbrace g^n_{i+\frac12},0\rbrace  \delta_{i+1} 
	\label{eq:flux_one_st}
\end{equation}
and the short notation $F^{n,\delta}_{g_{i+\frac{1}{2}}} = F^{\delta}_{g_{i+\frac{1}{2}}}(\bU^n).$
As mentioned earlier, the NSWs $\delta_\nu$ depend on the particular method. For example, considering an MPRK scheme as in \cite{kopecz_order_2018}, $\delta_\nu$ represent the corresponding Patankar weights, which are $\delta_i=\frac{U_i^{n+1}}{U_i^n}$ for the MPE. 

\begin{rmk}[Non uniqueness of the splitting]\label{rem:other_splittings} First of all, with the construction \eqref{eq:prod_dest}, there always exists a splitting of the conservation law into a conservative PDS. However, it should be noted that the splitting in production and destruction terms is not unique as, in general, one could add a positive constant to all production and destructions $p^{\text{new}}_{ij}\coloneqq p^{\text{old}}_{ij} + c$ for all $i\neq j$, and $d^{\text{new}}_{ij}\coloneqq p^{\text{new}}_{ji}$, leading to another valid splitting of the original PDS for any $c\geq 0$.

	Furthermore, there might be even multiple non-trivial splittings, e.\,g., for the inviscid Burgers' equation using an upwind FV discretization, where $U_i\geq 0, \,f(U_i)=\frac12 U_i^2 \geq 0$, we obtain
	\begin{equation*}
		\begin{aligned}
		\frac{\mathrm{d} U_i}{\mathrm{dt}} =&-\frac{1}{\Delta x}\left( f(U_{i})-f(U_{i-1})\right)	=
		\frac{1}{\Delta x}f(U_{i-1})-\frac{1}{\Delta x}f(U_i)\\
		=&\hphantom{+}\frac{1}{\Delta x}\left( -(f(U_{i})-f(U_{i-1}))^-+f(U_i)\right) \; -\frac{1}{\Delta x}(f(U_i)-f(U_{i-1}))^+\; \\
		&+\frac{1}{\Delta x}\left( (f(U_{i+1})-f(U_i))^-
		 -f(U_{i+1})\right) + (f(U_{i+1})-f(U_i))^+,
				\end{aligned}
	\end{equation*}
	where we introduced $\tau^+\coloneqq \max(\tau,0)$, $\tau^-\coloneqq \min(\tau,0)$ for a scalar function $\tau$, i.\,e., $\tau^++\tau^-=\tau$. Here, both re-formulations may be interpreted as different conservative PDS splittings satisfying $p_{i,j}=0$ for $j\notin\{i-1,i+1\}\subseteq\{1,\dotsc,N\}$ , $d_{i,j}=p_{j,i}$ for all $i,j$, and
	\begin{equation*}
		p_{i,i-1}(\bU)=\frac{1}{\Delta x}f(U_{i-1}),\qquad d_{i,i+1} = -\frac{1}{\Delta x}f(U_i) ,\qquad 	p_{i,i+1}(\bU)=d_{i,i-1}=0\qquad \text{ or }
		\end{equation*}
	\begin{align*}
		p_{i,i-1}(\bU)&=\frac{1}{\Delta x}\left( -(f(U_{i})-f(U_{i-1}))^-+f(U_i)\right),\qquad d_{i,i-1}(\bU)=\frac{1}{\Delta x}(f(U_i)-f(U_{i-1}))^+,\\
		d_{i,i+1}(\bU)&=-\frac{1}{\Delta x}\left( (f(U_{i+1})-f(U_i))^-
		 -f(U_{i+1})\right), \qquad p_{i,i+1}(\bU)= \frac{1}{\Delta x}(f(U_{i+1})-f(U_i))^+,
	\end{align*}
	respectively.
	It is worth noting that the first splitting corresponds to \eqref{eq:prod_dest} and possesses differentiable production and destruction terms in this example, while the second splitting does not. Also, by comparing $p_{i,i+1}$ it becomes evident that the two splittings differ. In what follows, we require that any splitting of the PDE results in a conservative PDS consisting of production and destruction terms which are non-negative on their domain.

\end{rmk}

Although, we first introduced one stage methods, also an $s$-stage Patankar-type difference scheme based on an explicit RK scheme with $\bA = (a_{kr})_{k,r=1,\dots,s}\geq \b 0$ and weights $b_1, \dots, b_s\geq 0$ can be devised (using $p_{i,j}=d_{j,i}$) by
\begin{equation}
	\begin{split}
		&U_i^{(k)}\!\! = U_i^n - \frac{\Delta t}{\Delta x} \sum\nolimits_{r=1}^{k-1} a_{kr} \biggl(d_{i,i+1}(\bU^{(r)}) \gamma_{i}^{(k)}  - p_{i,i+1}(\bU^{(r)}) \gamma_{i+1}^{(k)} -\bigl( d_{i-1,i}(\bU^{(r)}) \gamma_{i-1}^{(k)}-p_{i-1,i}(\bU^{(r)}) \gamma_{i}^{(k)} \bigr) \biggr), \\
		&U_i^{n+1}\!\! = U_i^n - \frac{\Delta t}{\Delta x} \sum\nolimits_{r=1}^{s} b_{r} \biggl(d_{i,i+1}(\bU^{(r)}) \delta_i  - p_{i,i+1}(\bU^{(r)}) \delta_{i+1}-\bigl( d_{i-1,i}(\bU^{(r)}) \delta_{i-1}-p_{i-1,i}(\bU^{(r)}) \delta_{i} \bigr) \biggr) ,
	\end{split}
	\label{eq:NSARK_stages}
\end{equation}
for $k=1,\dots,s$, where new NSWs $\gamma_\nu^{(k)}= \gamma_\nu^{(k)}(\hat\bU^n,t_n, \Delta t)$ are introduced for the stages $U_i^{(k)}$.
We define the new numerical flux based on $F^{\delta}_{g_{i+\frac{1}{2}}}(\bU^n)$ from \eqref{eq:flux_one_st}  as
\begin{equation}
	\bF^{n,\delta}_{g_{i+\frac{1}{2}}} := \sum\nolimits_{r = 1}^{s}b_r F^{\delta}_{g_{i+\frac{1}{2}}}(\bU^{(r)}) .
	\label{eq:sum_flux}
\end{equation}

\begin{rmk}[Systems of conservation laws]
	This approach can also be extended to systems of conservation laws.
	There, we apply the MP trick on every component of the vector $\mathbf u$ of the initial value problem that requires positivity, and we consider the numerical flux  for every such component as a scalar function $g$. With that it becomes evident that this approach can be applied 
    also for different space-discretizations as long as they can be written in flux formulation \cite{abgrall2023}.
\end{rmk}

\begin{rmk}[Negative RK weights]
    The MPRK can be designed also when some of the $\mathbf{A}$ and $\mathbf b$ coefficients are negative, see \cite{kopecz_order_2018}, by simply changing the role of production and destruction for that particular term in that stage. This has been applied systematically to design arbitrary high order MP methods as the MPDeC \cite{offner2020arbitrary} that will be used in the numerical section.
\end{rmk}

\begin{rmk}[High order discretizations]
    The difference schemes framework \eqref{eq:diffscheme} is very general and can encompass various spatial discretizations. This includes also various high order discretizations as Essentially Non-Oscillatory (ENO), WENO, FV \cite{ciallella2022arbitrary,ciallella2025high} and finite difference \cite{huang_positivity-preserving_2019} reconstructions schemes. 
    A slightly different formulation should be introduced to include also other high order discretizations as continuous and DG FE methods \cite{MO2014,zhu2024bound}, where the PDS is either written only on the interfaces of the cells or there is more freedom in the choice of the PD terms between the degrees of freedom inside each cell. 
\end{rmk}

\section{Lax--Wendroff theorem for MP difference schemes}\label{se_Main}

The aim of this section is to adapt the proof of Theorem~\ref{thm:LW} to Patankar-type schemes in the form \eqref{eq:NSARK_stages}. To achieve this, we need to control the nonlinear weights, resulting in the extra assumption that the numerical solution has a bounded total variation also in time (classical Lax--Wendroff requires only bounded total variation in space).
So, let us introduce the concept of total time variation (TTV). 
\begin{dfn}[Total time variation]\label{def:TTV}
    The TTV over a domain $[a,b]\times [0,T]$ of a function $U \in L^1_{\text{loc}}(\mathbb R \times \mathbb R^+)$ is 
    \begin{equation}
        \text{TTV}(U(x,[0,T]))\coloneqq \sup \sum_{n = 1}^{n_T}|U(x,t_n) -U(x,t_{n-1})|,
    \end{equation}
    where the supremum is taken over all possible subdivisions $t_0=0 <\dots < t_{n_T} = T$ for any $n_T\in \mathbb N$.
\end{dfn}

\begin{rmk}[Discrete total time variation]\label{rmk:TTV}
    On a domain $[a,b]\times [0,T]$ for a discrete function $\lbrace U_i^n \rbrace_{i,n}=\{U(x_i,t_n)\}_{i,n}$ for $i=1,\dots,N$ and $n=0,\dots, n_T$ with $t_0=0$ and $t_{n_T}=T$, the TTV can be equivalently defined using a piece-wise constant reconstruction of the discrete function on the space-time grid, obtaining the formula  
    \begin{equation}\label{eq:TTV}
        \text{TTV}(U(x,[0,T]))\coloneqq\sum_{n = 1}^{n_T}|U^{n}(x) -U^{n-1}(x)|,
    \end{equation}
    with $U^n(x)$ being the piece-wise constant reconstruction in space of $\lbrace U^n_i\rbrace_i$.
\end{rmk}
It should be noticed that this definition differs from the definition of total of time variation (TOTV) given in \cite{toth2023total}, and it is the equivalent of the TV definition in space, but switching the roles of time and space.

\begin{rmk}[TTV uniform boundedness]
	The TTV uniform boundedness, i.e., having a bound for the TTV of an approximate solution while the mesh is refined, will be used as a requirement in the next sections in many theorems. As for the TV boundedness in space, it is related to the non-oscillatory behavior of the numerical solution in time and, hence, their stability. 
	Moreover, we believe that there is a link between the TV boundedness and the TTV boundedness for hyperbolic problems, though it is not strainghtforward to prove it.
	During numerical results, we have seen that TTV boundedness seems naturally achieved 
	(with $\text{TTV}(U^{num}(x))$
	 converging to the exact $\text{TTV}(U^{ex}(x)) $) 
	 when a stable CFL number is used. On the contrary, when oscillations arise in the space discretization, they also propagate in time, leading to not converging $\text{TTV}(U^{num}(x))$. 
\end{rmk}

\begin{lem}[MPE weights convergence]\label{lem:delta}
Let $U^{\{l\}}(x_i,t_n)$ be the approximation of problem $\eqref{eq:InitProb}$ at the space $x^{\{l\}}_i$ and time $t_n^{\{l\}}$ for a mesh discretization $\Delta x_{\{l\}}\sim \Delta t_{\{l\}}\to 0$ for $l\to \infty$. For every $l$, consider $U^{\{l\}}(x,t)$ as a piece-wise constant reconstruction of the numerical solution $U^{\{l\}}(x_i,t_n)$. Let $\TTV(U^{\{l\}}(x,[0,T]))$ be uniformly bounded for all $x\in [a,b]$ and for all $l\in \mathbb N$. Suppose that $U^{\{l\}}\to u$ in $L^1_{\text{loc}}([a,b]\times [0,T] )$ with $u$ positive. Defining $\delta^{\{l\}}(x,t)$ as the piece-wise constant reconstruction of the MPE weights $\delta^{\{l\}}(x_i,t_n)=U^{\{l\}}(x_i,t_{n+1})/U^{\{l\}}(x_i,t_n)$ for all $i$ and $n$. Then, $\delta^{\{l\}}(x,t) \to 1$ almost everywhere (a.e.) in $[a,b]\times [0,T]$.
\end{lem}

\begin{proof}
	Consider the MPE weights 
	$$\delta^{\{l\}}(x,t):= \frac{U^{\{l\}}(x,t+\Delta t_{\{l\}})}{U^{\{l\}}(x,t)}=1+\frac{U^{\{l\}}(x,t+\Delta t_{\{l\}})-U^{\{l\}}(x,t)}{U^{\{l\}}(x,t)}.$$
	Now, we can say that $\lim_{l\to\infty} U^{\{l\}}(x,t+\Delta t_{\{l\}})-U^{\{l\}}(x,t) = 0$ a.\,e.\ because the numerical solutions are TTV bounded uniformly for all $l$ and $x$ and the denominator, $U^{\{l\}}$ is converging to a weak solution $u$ in $L^1_{\text{loc}}$.
\end{proof}
The arguments of this proof are the same used in space for the Lax--Wendroff theorem in the proof of \cite{leveque_numerical_1992}. These follow from some classical analysis results using the compactness of bounded total variation functions in $L^1_{loc}$, the fact that in 1D functions of bounded variation have at most countably many discontinuities and the Helly’s selection theorem to deduce that bounded sequences of 1D bounded total variation functions are precompact for pointwise convergence almost everywhere \cite{ambrosio2000functions}.

\begin{thm}[Lax--Wendroff for MPE difference schemes]\label{th_Main_MP}
    Consider a sequence of step sizes $\Delta x_{\{l\}} > 0$ with  $ \lim_{l{ \rightarrow \infty}} \Delta x_{\{l\}} = 0$. Let $\left( U^{\{l\}} \right)_l $ be a sequence of piece-wise constant functions constructed of numerical solutions obtained by a MPE difference scheme \eqref{eq:MPRK_difference} consistent with \eqref{eq:InitProb} with time step size $\Delta t_{\{l\}} = \lambda \Delta x_{\{l\}}, \lambda \equiv \mathrm{const}.$ If for any domain $[a,b]\times [0,T]$ the total variation $TV \left(U^{\{l\}}([a,b], t)\right)$ from \eqref{eq:TTV} is uniformly bounded for all $l$ and $t \in [0,T]$, the total time variation $TTV\left(U^{\{l\}}(x, [0,T])\right)$ is uniformly bounded for all $l$ and $x \in [a,b]$ and if $U^{\{l\}}\xrightarrow{l \rightarrow \infty} u$ in $L^1_{\mathrm{loc}}(\R\times [0,T))$, then $u$ is a weak solution of the conservation law \eqref{eq:InitProb}.
    \label{thm:LWMP}
\end{thm}
\begin{proof}
		We follow the original Lax--Wendroff proof from \cite{lax_systems_1960,leveque_numerical_1992} and adapt it for our case where needed. The main goal is to show that the limit function $u(x,t)$ of the considered method satisfies
	\begin{equation}
		\int_{0}^{\infty} \int_{-\infty}^{\infty} [\phi_t u + \phi_x f(u)]\dx \mathrm{dt} = - \int_{-\infty}^{\infty} \phi(x,0)u(x,0) \dx.
		\label{eq:weak_LW}
	\end{equation}
     for all $\phi \in C_0^1([a,b]\times [0,T])$. In what follows, we omit to write the dependency on the refinement represented by the index $l$ if it is clear from the context. 
     
Multiplying the numerical method \eqref{eq:MPRK_difference} by any test function  $\phi \in C_0^1$ and summing over all $i$ and $n\geq 0$ yields
	\begin{equation*}
		\sum_{n=0}^{\infty} \sum_{i=-\infty}^{\infty} \phi(x_i,t_n)(U_i^{n+1}-U_i^n)=\frac{\Delta t}{\Delta x} \sum_{n=0}^{\infty} \sum_{i=-\infty}^{\infty} \phi(x_i,t_n)\left[F^{\delta}_{g_{i+\frac{1}{2}}}(\bU^n)-F^{\delta}_{g_{i-\frac{1}{2}}}(\bU^n)\right].
	\end{equation*} We note that the value of, e.\,g., $U_i^n$ does not have an effect for $(x_i,t_n)\notin\supp(\phi)\subseteq[a,b]\times [0,T]$. Using summation by parts, the fact $\phi$ has a compact support, multiplying by $\Delta x$ and rearranging leads to
	\begin{equation}
    \begin{split}
		&\Delta x\Delta t\left[\sum_{n=1}^{\infty} \sum_{i=-\infty}^{\infty} \left(\frac{\phi(x_i,t_n)-\phi(x_i,t_{n-1})}{\Delta t}\right)U_i^n +\sum_{n=0}^{\infty}\sum_{i=-\infty}^{\infty}\left(\frac{\phi(x_{i+1},t_n)-\phi(x_i,t_n)}{\Delta x}\right)F^{\delta}_{g_{i+\frac{1}{2}}}(\bU^n)\right] \\
        =&-\Delta x\sum_{i=-\infty}^{\infty}\phi(x_i,0)U_i^0.
		\label{eq:disc_prf_lw}
        \end{split}
	\end{equation}
    Now we show that this equation coincides with the desired assertion \eqref{eq:weak_LW} as we pass to the limit $ \lim_{l{ \rightarrow \infty}} \Delta x_{\{l\}} = 0$ with $\Delta t_{\{l\}} = \lambda \Delta x_{\{l\}}, \lambda \equiv \mathrm{const}$. 
    For the first and the last term of the previous equation, the classical proof of Lax and Wendroff should be applied, as outlined in \cite{leveque_numerical_1992}. This means that as $l \rightarrow \infty$, the assumed $L^1_{\mathrm{loc}}$ convergence of $U^{\{l\}}$ to $u$, together with the smoothness of $\phi$, ensures the convergence of the first summand to $\int_0^\infty \int_{-\infty}^{\infty} \phi_t(x,t)u(x,t)\dx$. In addition, the choice of the initial data $U_i^0=\psi_i$ enables us to conclude the convergence of the RHS to $-\int_{-\infty}^{\infty} \phi(x,0)u(x,0)\dx$. 
	The remaining task is to ascertain the limit for the middle term. In particular, we are interested in showing the convergence of the new numerical flux to the physical flux and therefore in estimating 
    \begin{equation}
		\left\lvert F^{\delta,\{l\}}_{g} (\mathbf{x}_i,t_n) - f(U^{\{l\}}(x_i,t_n)) \right\rvert
        \label{eq:conv_numflux_lw}
	\end{equation}
	for $l\to \infty$
	where $\mathbf x_i = (x_i-p\Delta x,\dots, x_i+q\Delta x)$. Here, $F^{\delta, \{l\}}_{g} (\mathbf{x}_i,t_n)$ is a piece-wise constant reconstruction of the values $F^{\delta}_{g_{i+\frac{1}{2}}}(\bU^n)$ given in \eqref{eq:flux_one_st} by substituting 
    \[
    g_{i+\frac12}^{n,\{l\}}
    =g(U^{\{l\}}(\mathbf{x}_i ,t_n))  
    = g\bigl(U^{\{l\}}(x_i-p \Delta x,t_n),\dots,U^{\{l\}}(x_i+q\Delta x,t_n)\bigr)
    \] for $i=1,\dots,N_l$ and $n=0,\dots,n_{T,\{l\}}$ for $g_{i+\frac12}^n$. Adding and subtracting
    \[
    g(U^{\{l\}}(\mathbf{x}_i ,t_n))=\max \lbrace g(U^{\{l\}}(\mathbf{x}_i ,t_n)),0 \rbrace  + \min \lbrace g(U^{\{l\}}(\mathbf{x}_i ,t_n)),0\rbrace 
    \]
to \eqref{eq:conv_numflux_lw} leads to
    \begin{align*}
		&\lvert F^{\delta}_{g_{i+\frac{1}{2}}}(\bU^n) - f(U^{\{l\}}(x_i,t_n)) \rvert= \left\lvert \max \lbrace g(U^{\{l\}}(\mathbf{x}_i ,t_n)),0 \rbrace \delta_{i}^{\{l\}} + \min \lbrace  g(U^{\{l\}}(\mathbf{x}_i ,t_n)),0\rbrace  \delta_{i+1}^{\{l\}}- f(U^{\{l\}}(x_i,t_n)) \right\rvert \\
	&\leq  \left\lvert \max \lbrace g(U^{\{l\}}(\mathbf{x}_i ,t_n)),0 \rbrace \delta_{i}^{\{l\}} -\max \lbrace g(U^{\{l\}}(\mathbf{x}_i ,t_n)),0 \rbrace  \right\rvert +  \left\lvert \min \lbrace g(U^{\{l\}}(\mathbf{x}_i ,t_n)),0\rbrace  \delta_{i+1}^{\{l\}}-\min \lbrace g(U^{\{l\}}(\mathbf{x}_i ,t_n)),0\rbrace \right\rvert  \\
        &\qquad + \left\lvert \max \lbrace g(U^{\{l\}}(\mathbf{x}_i ,t_n)),0 \rbrace  + \min \lbrace g(U^{\{l\}}(\mathbf{x}_i ,t_n)),0\rbrace  - f(U^{\{l\}}(x_i,t_n)) \right\rvert\\
        &= \left\lvert \max \lbrace g(U^{\{l\}}(\mathbf{x}_i ,t_n)),0 \rbrace \delta_{i}^{\{l\}} -\max \lbrace g(U^{\{l\}}(\mathbf{x}_i ,t_n)),0 \rbrace  \right\rvert +  \left\lvert \min \lbrace g(U^{\{l\}}(\mathbf{x}_i ,t_n)),0\rbrace  \delta_{i+1}^{\{l\}}-\min \lbrace g(U^{\{l\}}(\mathbf{x}_i ,t_n)),0\rbrace \right\rvert  \\
        &\qquad+ \left\lvert g(U^{\{l\}}(\mathbf{x}_i ,t_n))  - f(U^{\{l\}}(x_i,t_n)) \right\rvert,
	\end{align*}
	where the first two addends vanish a.e. as $l \rightarrow \infty$ due to the convergence of the MPE weights based on the TTV boundedness property, as stated in Lemma~\ref{lem:delta}.
	
    Finally, as only the original numerical flux $g$ appears in the remaining term $\left\lvert g(U^{\{l\}}(\mathbf{x}_i ,t_n))  - f(\bU^{\{l\}}(x_i,t_n)) \right\rvert$, we can follow the reasoning of the Lax--Wendroff theorem of \cite{leveque_numerical_1992}. 
Consistency  of $g$ is provided, i.e., we have the estimate
    \begin{align*}
    &\left\lvert g(U^{\{l\}}(\mathbf{x}_i ,t_n))  - f(U^{\{l\}}(x_i,t_n)) \right\rvert  = \\
    &\left\lvert g\bigl(U^{\{l\}}(x_i-p \Delta x,t_n),\dots,U^{\{l\}}(x_i+q\Delta x,t_n)\bigr)  - f(U^{\{l\}}(x_i,t_n)) \right\rvert  
    \leq K \max_{-p\leq j \leq q} \left\lvert U^{\{l\}}(x_i + j\Delta x, t_n) - U^{\{l\}}(x_i,t_n)\right\rvert        
    \end{align*}
    with a Lipschitz constant $K$, where the total variation boundedness (TVB) property of $U^{\{l\}}(\cdot,t)$ leads to
    $$\max_{-p\leq j \leq q} \left\lvert U^{\{l\}}(x + j\Delta x, t) - U^{\{l\}}(x,t)\right\rvert\rightarrow 0$$ as $l \rightarrow \infty$ for almost every value of $x$.
    Altogether we proved the a.\,e.\ convergence of the numerical flux function $F^{\delta, \{l\}}_{g} (\mathbf x_i,t_n)$ to the physical flux function $f(U^{\{l\}}(x_i,t_n))$. 
    This gives, together with some standard estimates, the convergence of \eqref{eq:disc_prf_lw}  to the weak form \eqref{eq:weak_LW} and shows, since it is true for any test function $\phi \in C_0^1$, that $u$ is indeed a weak solution.

\end{proof}

\subsection{Extension to Patankar-type RK}
In this section, we draw the steps to extend the previous proof to the multi-stage Patankar-type schemes.
\begin{dfn}[Total time variation for Runge--Kutta]
	The total time variation of a numerical solution $\lbrace U_i^n \rbrace_{i,n}=\{U(x_i,t_n)\}_{i,n}$ obtained by an s-stage RK method (TTVRK) over a domain $[a,b]\times [0,T]$ for $i=1,\dots,N$ and $n=0,\dots, n_T$ with $x_1=a$, $x_N=b$, $t_0=0$ and $t_{n_T}=T$ is 
    	\begin{equation}
				\text{TTVRK}(\bU(x,[0,T])):=\sum_{n = 0}^{n_T}\sum_{k=1}^s|U^{(k)}(x,t_n) -U^{(k-1)}(x,t_n)| +\sum_{n = 1}^{n_T}|U(x,t_n) -U^{(s)}(x,t_{n-1})|,
	\end{equation}
    where $U^{(0)}(x,t_n)\coloneqq U(x,t_n)$ and $U^{(k)}(x,t_n)$ denotes the piece-wise constant reconstruction in space of the $k$-th stage computed from the numerical solution $\mathbf U^n$ at time level $n$.
\end{dfn}

\begin{lem}[Extended Patankar weights convergence]\label{lem:weights_MPRK}
	Let $U^{\{l\}}(x_i,t_n)$ be the approximation of problem $\eqref{eq:InitProb}$ at the space $x^{\{l\}}_i$ and time $t_n^{\{l\}}$ for a mesh discretization $\Delta x_{\{l\}}\sim \Delta t_{\{l\}}\to 0$ for $l\to \infty$. For every $l$, consider $U^{\{l\}}(x,t)$ a piece-wise constant reconstruction of the numerical solution $U^{\{l\}}(x_i,t_n)$. Let $\text{TTVRK}(U^{\{l\}}(x,[0,T]))$ be uniformly bounded for all $x\in [a,b]$ and for all $l\in \mathbb N$. Suppose that $U^{\{l\}}\to u$ in $L^1_{\text{loc}}([a,b]\times [0,T] )$. Defining $\delta^{\{l\}}(x,t)$ as the piece-wise constant reconstruction of $\delta^{\{l\}}(x_i,t_n)=U^{\{l\}}(x_i,t_{n+1})/\sigma^{\{l\}}(x_i,t_n)$ and $\gamma^{(k),\{l\}}(x,t)$ the one of $\gamma^{(k),\{l\}}(x_i,t_n^k)=U^{(k),\{l\}}(x_i,t_{n}^k)/\pi^{(k),\{l\}}(x_i,t_n^k)$ for all $i$ and $n$. Then, $\delta^{\{l\}}(x,t)$ and $\gamma^{(k),\{l\}}(x,t) \to 1$ almost everywhere in $[a,b]\times [0,T]$ if $\sigma^{\{l\}}(x,t)$, respectively, $\pi^{(k),\{l\}}(x, t) \rightarrow u(x,t)$ as $l\to\infty$.
\end{lem}

\begin{proof}
First we investigate the weights of the update
$$\delta^{\{l\}}(x,t):= \frac{U^{\{l\}}(x,t+\Delta t_{\{l\}})}{\sigma^{\{l\}}(x,t)}=1+\frac{U^{\{l\}}(x,t+\Delta t_{\{l\}})-\sigma^{\{l\}}(x,t)}{\sigma^{\{l\}}(x,t)}.$$ 
	Now, we can follow the same reasoning as in the proof of Lemma~\ref{lem:delta} under usage of the uniformly bounded TTVRK and the convergence of the numerical solutions to a weak solution $u$ in $L^1_{\text{loc}}$, to show $\lim_{l\to\infty} U^{\{l\}}(x,t+\Delta t_{\{l\}})=u(x,t)$ a.e.. Furthermore we assumed that $\sigma^{\{l\}}(x,t) \rightarrow  u(x,t)$ which finally leads to $\lim_{l\to\infty} U^{\{l\}}(x,t+\Delta t_{\{l\}})-\sigma^{\{l\}}(x,t) = 0$ and thus to $\delta^{\{l\}}(x,t) \rightarrow 1$.
We proceed considering the weights of the stages
$$\gamma^{(k),\{l\}}(x,t)=\frac{U^{(k),\{l\}}(x,t)}{\pi^{(k),\{l\}}(x,t)} = 1 + \frac{U^{(k),\{l\}}(x,t)-\pi^{(k),\{l\}}(x,t)}{\pi^{(k),\{l\}}(x,t)}.$$
Once again, we use the fact that TTVRK is uniformly bounded, together with the convergence to a weak solution $u$ for the limit $\lim_{l\to\infty} U^{(k),\{l\}}(x,t)=u(x,t)$. 
Additionally, we assumed $\pi^{(k),\{l\}}(x, t) \rightarrow u(x,t)$ as $l\to\infty$, so this leads to $\lim_{l\to\infty} U^{(k),\{l\}}(x,t)-\pi^{(k),\{l\}}(x,t) = 0$ and thus to the desired result.
\end{proof}

In the previous Lemma~\ref{lem:weights_MPRK}, some convergence for the PWDs was assumed, that is $\sigma^{\{l\}}(x,t)$ respectively $\pi^{(k),\{l\}}(x,t) \rightarrow u(x,t)$ almost everywhere for $l \rightarrow \infty$ and $(x,t) \in [a,b]\times [0,T]$.

\begin{thm}[Lax--Wendroff for MPRK schemes]\label{thm:lxw_MPRK}
     Assume the PWDs of an MPRK method continuously depend on the $i$-th component of the stages, i.e.
    $$ \pi_i^{(k)} = \pi_i^{(k)}(U_i^n, U_i^{(1)}, \dots, U_i^{k-1}) \quad \text{and} \quad
    \sigma_i = \sigma_i(U_i^n, U_i^{(1)}, \dots, U_i^{(s)}).
    $$
    Furthermore we assume that the PWDs are consistent in the sense that 
    $$
    \pi_i^{(k)}(U_i^n, U_i^n, \dots, U_i^n) = U_i^n \quad \text{and} \quad 
    \sigma_i(U_i^n, U_i^n, \dots, U_i^n) = U_i^n.
    $$
    Then, the PWDs fulfill the extended Patankar weights convergence of Lemma~\ref{lem:weights_MPRK}.
\begin{proof}
    Recall that $U^{(k),\{l\}} \rightarrow u$ for $l \rightarrow \infty$ by means of the uniform boundedness of the TTVRK together with $U^{\{l\}}\to u$ in $L^1_{\text{loc}}([a,b]\times [0,T] )$, as mentioned in the proof of Lemma~\ref{lem:weights_MPRK}. Hence, the arguments of the PDWs tend to $u$, and thus, the continuity and consistency of the PWDs yield $\sigma^{\{l\}}(x,t),\,\pi^{(k),\{l\}}(x,t) \rightarrow u(x,t)$ almost everywhere for $l \rightarrow \infty$ and $(x,t) \in [a,b]\times [0,T]$.
\end{proof}
\end{thm}

\begin{rmk}
The assumptions made in the previous theorem are exactly those of \cite{IzginThesis} and \cite{avila_extension_2021} which cover common methods such as the ones presented in \cite{kopecz_order_2018,KM18Order3,offner2020arbitrary}.
     Altogether Theorem~\ref{thm:lxw_MPRK} shows that the PWDs from the literature already satisfy the assumption from Lemma~\ref{lem:weights_MPRK}.
\end{rmk}

In the following example, we illustrate that the assumption of Theorem~\ref{thm:lxw_MPRK} are fulfilled for a second-order family  of MP methods.

\begin{example}[Verification of PWD Hypothesis for MPRK22($\alpha$)]
    Consider the two stage second-order family of MPRK methods introduced in \cite{kopecz_order_2018} and denoted by MPRK22($\alpha$), where $\alpha\geq \frac12$ is a free parameter and the PWDs are given by $\pi_i^{(2)} = U_i^n$ and $\sigma_i = \left(U_i^n\right)^{1-\frac{1}{\alpha}}\left(U_i^{(2)}\right)^{\frac{1}{\alpha}}$.
    
    For $l \rightarrow \infty$ 
    we find $\gamma^{(2),\{l\}}(x,t) = \frac{U^{(2),\{l\}}(x,t)}{U^{n,\{l\}}(x,t)} \rightarrow 1$ and thus 
    $\sigma_i^{\{l\}} \rightarrow \left(U_i^n\right)^{1-\frac{1}{\alpha}}\left(U_i^n\right)^{\frac{1}{\alpha}} = U_i^n$ i.e. $\delta^{\{l\}} \rightarrow 1$ almost everywhere. 
\end{example}

\begin{rmk}[Lax--Wendroff proof: extension to Patankar-type RK sketch]\label{rmk:RK}
    The extension to RK methods with more stages differs only in the investigation of equation \eqref{eq:conv_numflux_lw}, where we have, instead of the MPE weighted numerical flux  \eqref{eq:flux_one_st}, the generalized version \eqref{eq:sum_flux}, i.e.,        $$\left\lvert\bar{F}_{g_{i+\frac12}}^{n,\delta,\{l\}}-f(U^{\{l\}}(x_i,t_n))\right\rvert$$
    which, using the definition of $\bar{F}_{g_{i+\frac12}}^{n,\delta,\{l\}}$ and $\sum_{r=1}^s b_r =1$, reads
    $$\left\lvert \sum_{r=1}^{s} b_r F_g^{\delta, \{l\}}\left(U^{(r),\{l\}}(\mathbf x_i,t_n^{(r)})\right)-f(U^{\{l\}}(x_i,t_n))\right\rvert \leq \sum_{r=1}^{s} \lvert b_r\rvert \left\lvert F_g^{\delta, \{l\}}\left(U^{(r),\{l\}}(\mathbf x_i,t_n^{(r)})\right)-f(U^{\{l\}}(x_i,t_n))\right\rvert.$$
    Thus, defining again $\mathbf x_i= (x_i-p\Delta x,\dots, x_i+q\Delta x)$, we take an arbitrary summand $r$ and decompose it analogously to the proof of Theorem~\ref{thm:LWMP}, which leads to
    \begin{align*}
        \left\lvert F_g^{\delta, \{l\}}\left(U^{(r),\{l\}}(\mathbf x_i,t_n^{(r)})\right)-f(U^{\{l\}}(x_i,t_n))\right\rvert \leq &\left\lvert \max \lbrace g(U^{(r),\{l\}}(\mathbf x_i,t_n^{(r)})),0 \rbrace \delta_{i}^{\{l\}} -\max \lbrace g(U^{(r),\{l\}}(\mathbf x_i,t_n^{(r)})),0 \rbrace  \right\rvert \\
        + &\left\lvert \min \lbrace g(U^{(r),\{l\}}(\mathbf x_i,t_n^{(r)})),0\rbrace  \delta_{i+1}^{\{l\}}-\min \lbrace g(U^{(r),\{l\}}(\mathbf x_i,t_n^{(r)})),0\rbrace \right\rvert \\
        + &\left\lvert g(U^{(r),\{l\}}(\mathbf x_i,t_n^{(r)}))  - f(U^{\{l\}}(x_i,t_n)) \right\rvert.
    \end{align*}
    The first two lines vanish reasoned by $\delta^{\{l\}}_i\rightarrow 1$ for $l \rightarrow \infty$, i.e., we need to prove that in the limit $l \rightarrow \infty$ the last summand
    \begin{align*}
        &\left\lvert g(U^{(r),\{l\}}(\mathbf x_i,t_n^{(r)}))  - f(U^{\{l\}}(x_i,t_n)) \right\rvert \\
        =& \left \lvert g(U^{(r),\{l\}}( x_i-q\Delta x,t_n^{(r)}),\dots,U^{(r),\{l\}}( x_i+p\Delta x,t_n^{(r)}))   \right\rvert\leq L \max_{-p\leq j \leq q} \left\lvert U^{(r),\{l\}}(x_i + j\Delta x ,t_n^{(r)}) - U^{\{l\}}(x_i,t_n)\right\rvert \rightarrow 0,
    \end{align*}
    for a.e. $x_i$,
    where, independently of the argument, we can apply the Lipschitz continuity of $g$ with Lipschitz constant $L$ to get the estimation.
    
    Note in this case the RK stages are arguments of the numerical flux. Therefore, we first need the TTVRK boundedness property to apply Lemma~\ref{lem:weights_MPRK} and ensure that the stages converge to the same solution $\lim_{l\to\infty}U^{(r),\{l\}}(x_i + j\Delta x ,t_n^{(r)}) - U^{\{l\}}(x_i + j\Delta x, t_n)=0$ and proceed with the proof of the original Lax--Wendroff theorem to obtain the desired result.
\end{rmk}

\begin{rmk}[Lax--Wendroff for NSARK]
In the case of MP methods we assumed in Lemma~\ref{lem:weights_MPRK} and Theorem~\ref{thm:lxw_MPRK} that the PWDs fulfill certain properties to conclude that the NSW grid functions $\bm{\delta}^{\{l\}},\bm\gamma^{(k),\{l\}}$ satisfy
\begin{equation}\label{eq:NSW->1}
    \lim_{l\to\infty}\bm{\delta}^{\{l\}}(x,t)=1\quad \text{ and }\quad \lim_{l\to\infty}\bm\gamma^{(k),\{l\}}(x,t)=1
\end{equation}almost everywhere. This property represents the main ingredient in the proof outlined in Remark~\ref{rmk:RK}. Indeed, to extend the Lax--Wendroff theorem also to general NSARK schemes with general solution-dependent coefficients of the form \eqref{eq:NSW}, it is sufficient to add the condition \eqref{eq:NSW->1} to the classical assumptions of the Lax--Wendroff theorem. In the case of MPRK schemes, we used the TTVRK boundedness together with special properties of the PWDs (see Theorem~\ref{thm:lxw_MPRK}) to deduce \eqref{eq:NSW->1} in Lemma~\ref{lem:weights_MPRK}.
\end{rmk}

\begin{rmk}[On the TTV boundedness assumption]
    The difference with respect to the classical Lax--Wendroff theorem consists of the assumption that the numerical solutions are TTV bounded uniformly in space and in the limit process. The authors could not easily remove this assumption to prove the Lax--Wendroff theorem.
    Different strategies have been pursued but all of them ended up requiring some extra hypotheses.
    In particular, the technique proposed by Birken and Linders \cite{birken_conservation_2022} shows that, using iterative methods for implicit schemes, the Lax--Wendroff theorem can be applied also in the context of Patankar-type methods. Nevertheless, in their work the case where both the limit of iterations of the linear solver and the mesh refinement level go to infinity is not treated. In particular, it is hard to prove the $\lim_{l \to \infty} \lim_{z \to \infty}$  case, where $l$ is the mesh refinement level and $z$ is the linear solver iteration. Indeed, this leads to a dependency of the composed numerical flux in one step of the RK on the whole domain. This destroys the argument used in the classical Lax--Wendroff theorem to prove the flux convergence almost everywhere.
\end{rmk}

The hypotheses of being TVB in space and time for a given spatial and temporal discretization is not trivial to study, as one needs to find a bound uniformly with the mesh refinement, which in a discrete setting is virtually impossible to check.
What one can do in some cases is to prove a stronger result, i.e., the total variation diminishing (TVD) (in space) property.
In \ref{sec:TVD}, we study the TVD property of the MPE scheme applied to an upwind first order FV spatial discretization for the nonviscous Burgers' equation. We also propose some numerical tests to validate the findings.

\section{Numerical results}\label{se_numerics}

In this section, we will run some simple Riemann problem tests to validate the convergence towards exact solutions. What we are looking for is to test the regimes where the MP weights are large and change significantly the scheme with respect to explicit ones. Clearly, this happens close to discontinuity and close to 0 values of the solution, while everywhere else the MP weights quickly converges to 1. Hence, we choose to test Riemann problems with very low value of one of the two sides.

We will start from first order methods with the MPE time discretization and then we will proceed with higher order discretizations.  

\subsection{First order tests}
\subsubsection{Burgers' equation}
We start the discussion with Burgers' equation 
\begin{equation}\label{eq:Burgers}
\partial_t u +\partial_x \left(\frac{u^2}{2}\right)=0,\ \qquad x \in [-1,1],\, t\in[0,T],
\end{equation}
with a double Riemann problem for initial conditions, i.e.,
\begin{equation}\label{eq:RP}
    u(x,0) = \begin{cases}
        u_1, &\text{ if }-0.5<x<0.5,\\
        u_2, &\text{ else},
    \end{cases}
\end{equation}
with both $u_1,u_2>0$ and periodic boundary conditions. The exact solution of this problem if $u_1>u_2$ is a discontinuity traveling at speed $c = \frac{u_1^2-u_2^2}{2(u_1-u_2)}$ on the right and a rarefaction wave on the left.
For the spatial discretization, we use a simple upwind numerical flux that, given the positivity of the solution, will sum up to $g_{i+\frac12}(U^n) = U^n_{i}$.
In \ref{sec:TVD}, we show that MPE with the upwind spatial discretization is TVD for an implicit CFL bound: $\text{CFL}<2$.

To check the convergence of the scheme, we look at a very strong Riemann problem where $u_1=10^4$ and $u_2=10^{-30}$, so that we expect the MP weights to be particularly large. 
We run some simulations with different number of cells $N$ at CFL = 1 and CFL = 2.1. For each of these simulations we look for the interface with largest derivative and we set this point to be the ``numerical shock location'' $s_N$. We measure then $|s_N-s|$, where $s$ is the exact shock location. For both CFLs (so, also in a non TVD regime), we obtain that the numerical shock locations converge linearly to the exact shock location. 

\begin{figure}
    \centering
    \includegraphics[width=0.48\linewidth]{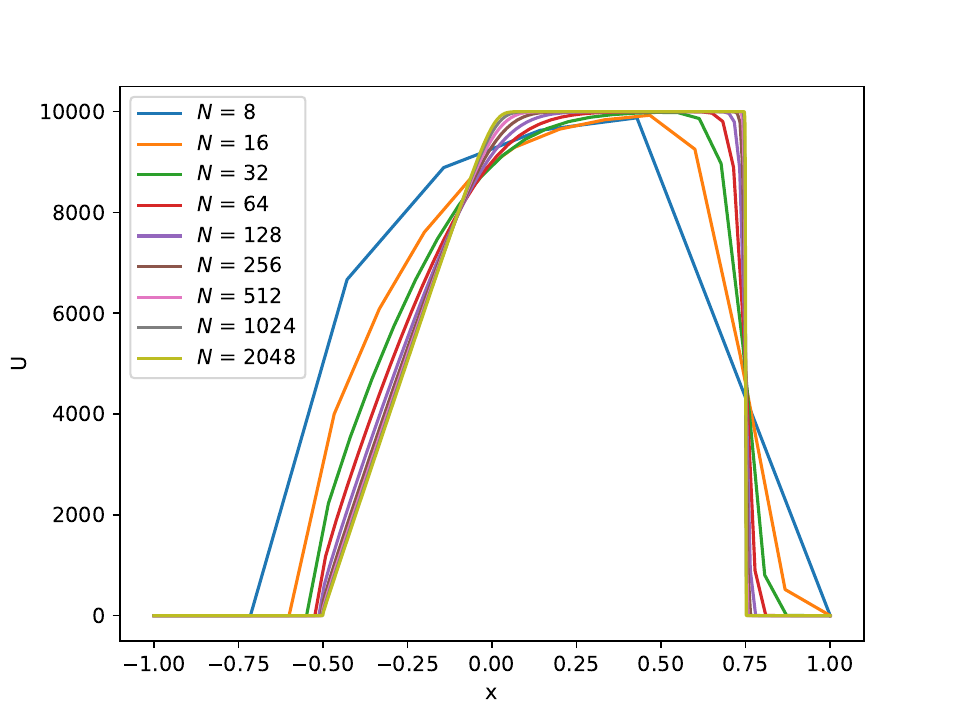}
    \includegraphics[width=0.48\linewidth]{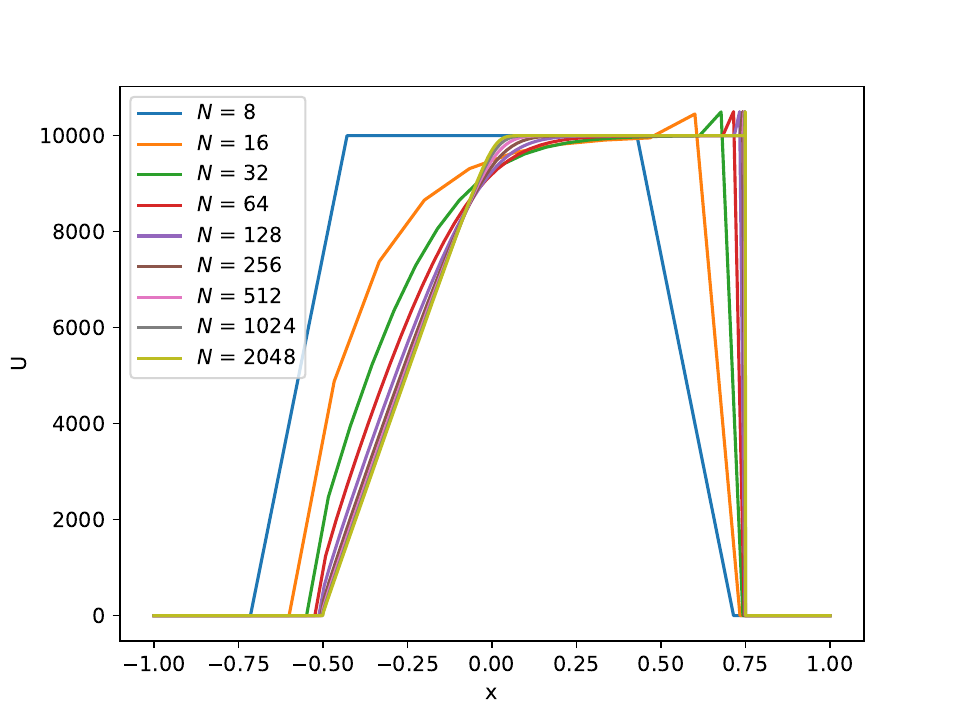}\\
    \includegraphics[width=0.48\linewidth]{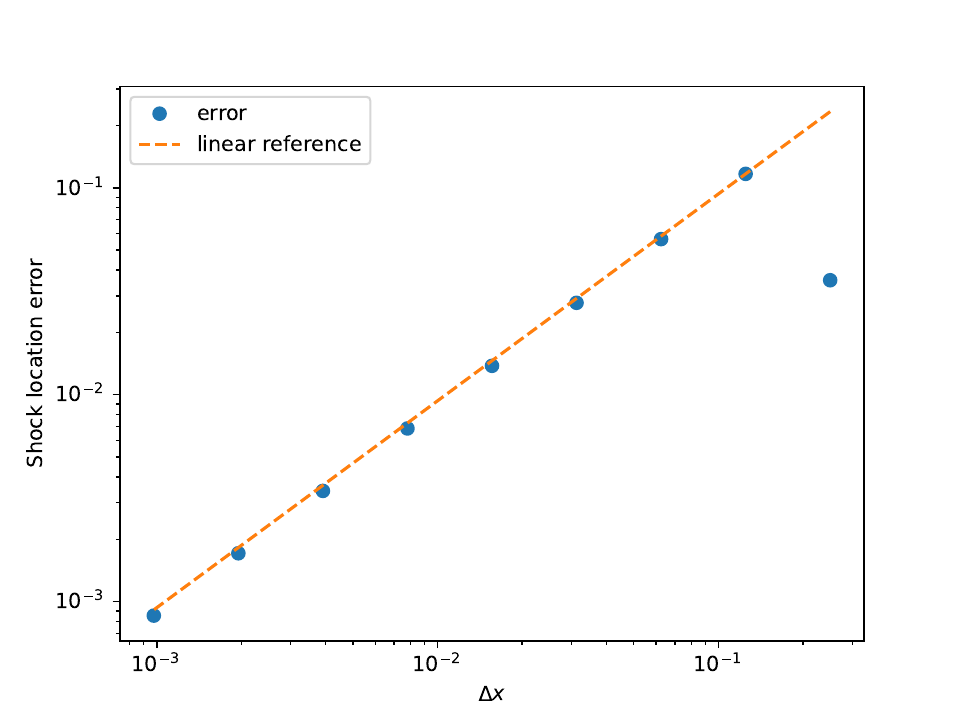}
    \includegraphics[width=0.48\linewidth]{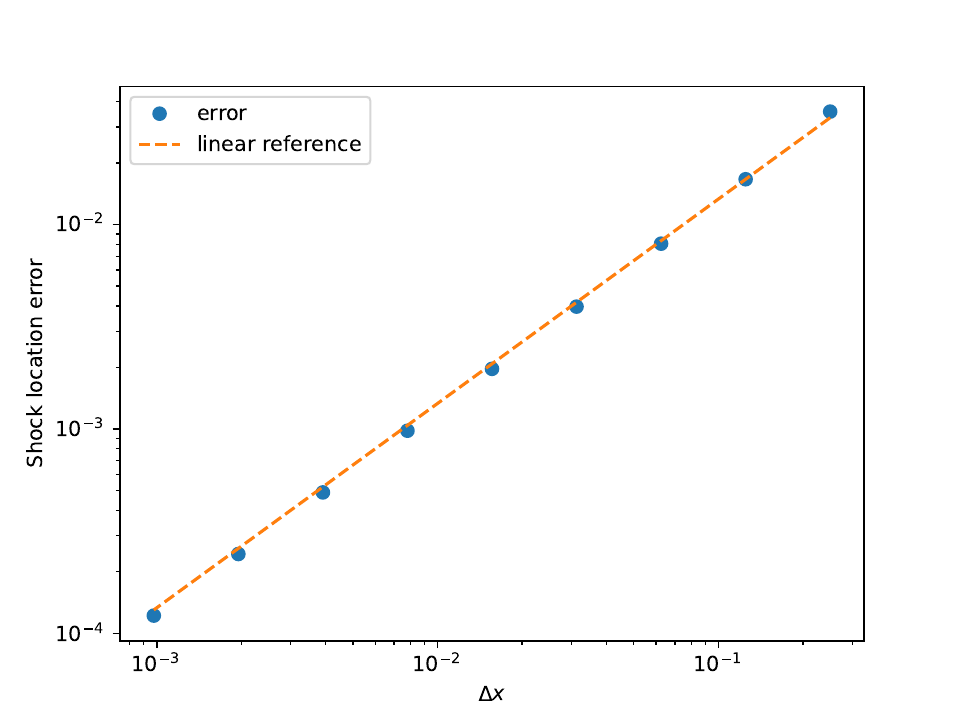}
    \caption{Simulations (top) and numerical shock location errors (bottom) for Burgers' equation \eqref{eq:Burgers} with upwind numerical flux and MPE. CFL=1 on the left and CFL=2.1 on the right}
    \label{fig:burgers_shock_location}
\end{figure}

In Figure~\ref{fig:burgers_shock_location}, we show the simulations and the error of the shock locations, proving the convergence of the numerical solutions to the exact solution of the problem.

\subsubsection{Buckley--Leverett equation}
\begin{figure}
    \centering
    Explicit Euler\\
    \includegraphics[width=0.32\linewidth]{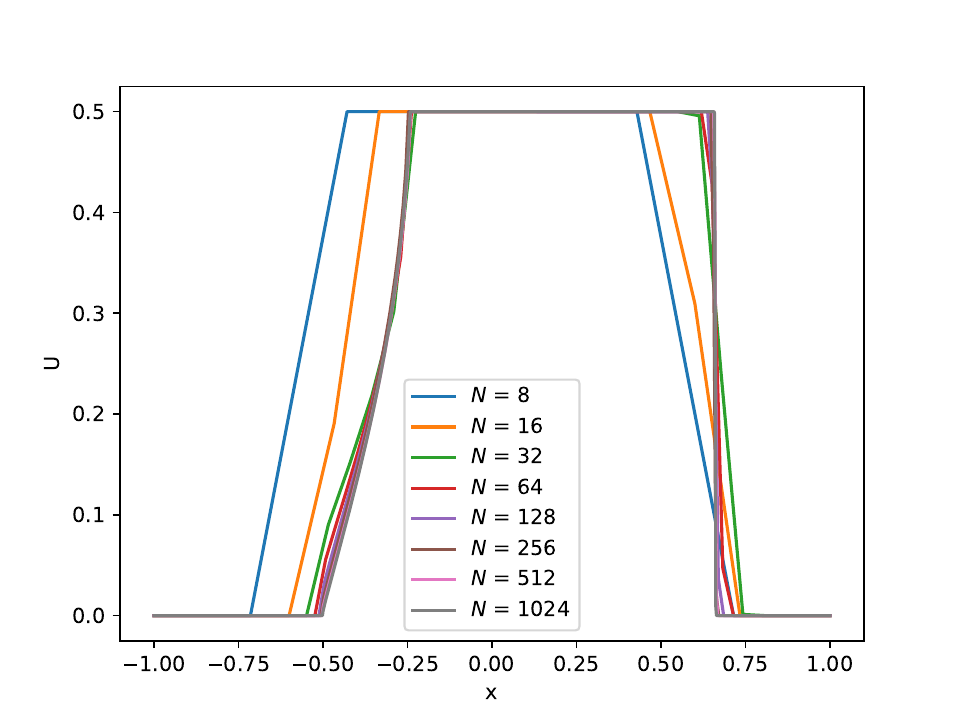}
    \includegraphics[width=0.32\linewidth]{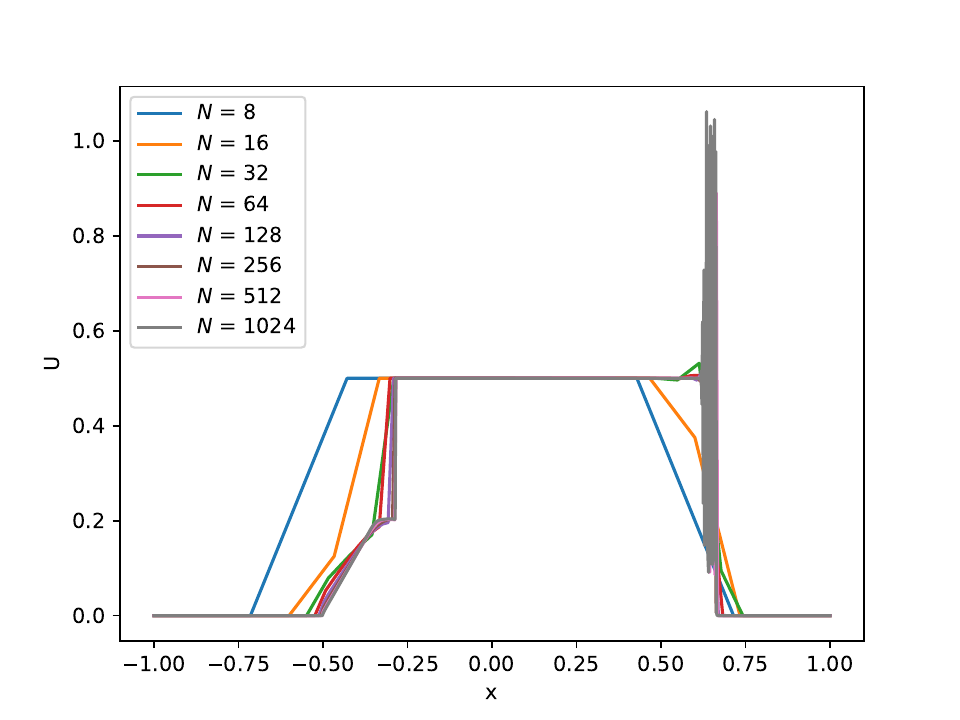}
    \includegraphics[width=0.32\linewidth]{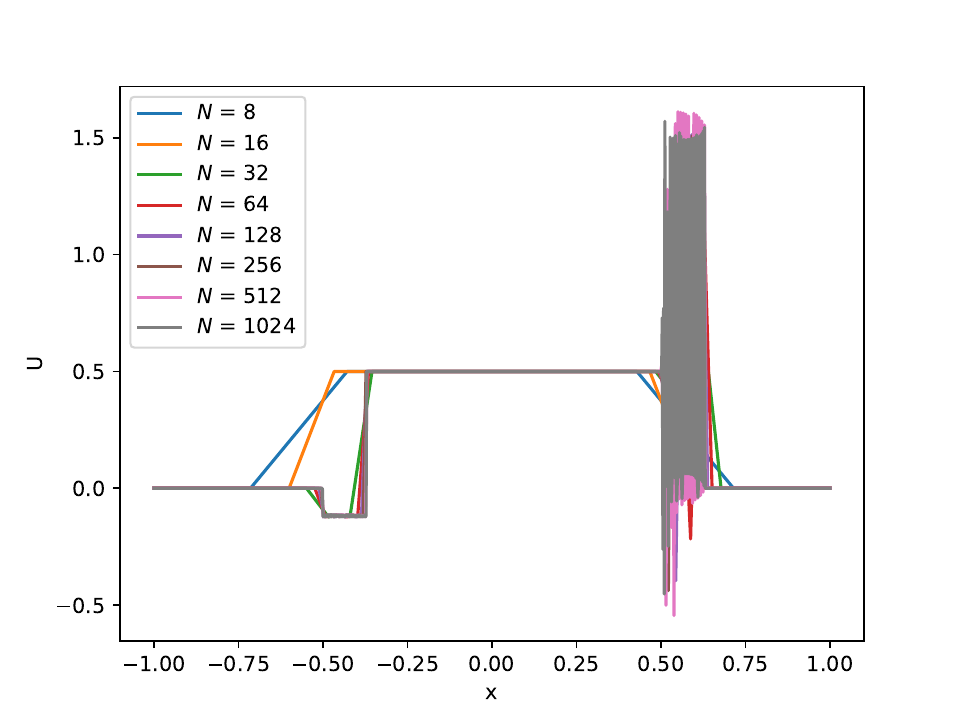}\\
    \includegraphics[width=0.32\linewidth]{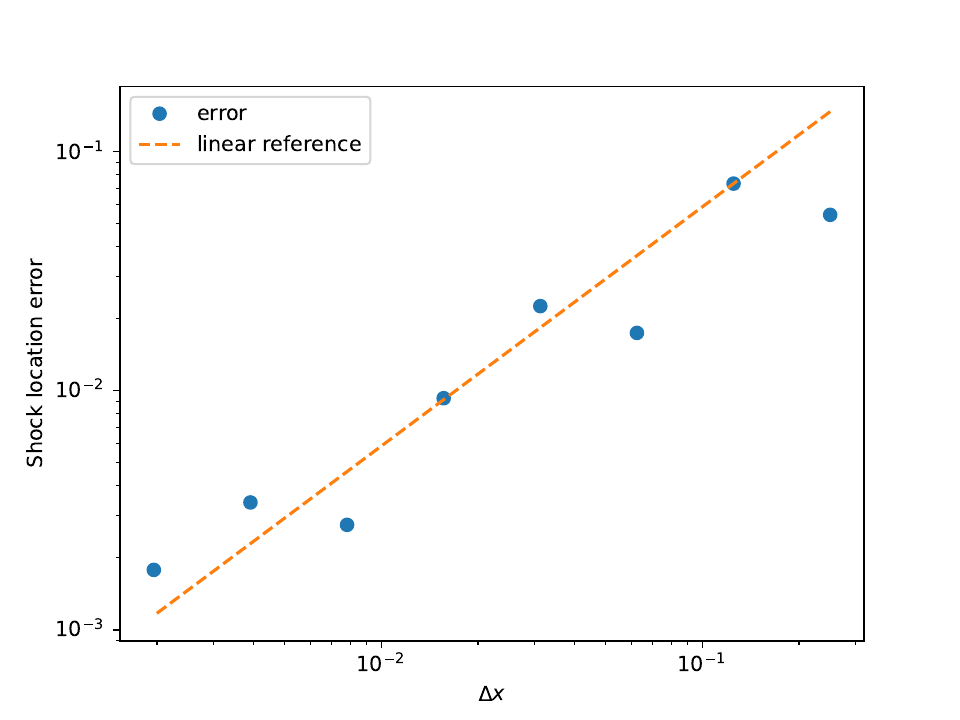}
    \includegraphics[width=0.32\linewidth]{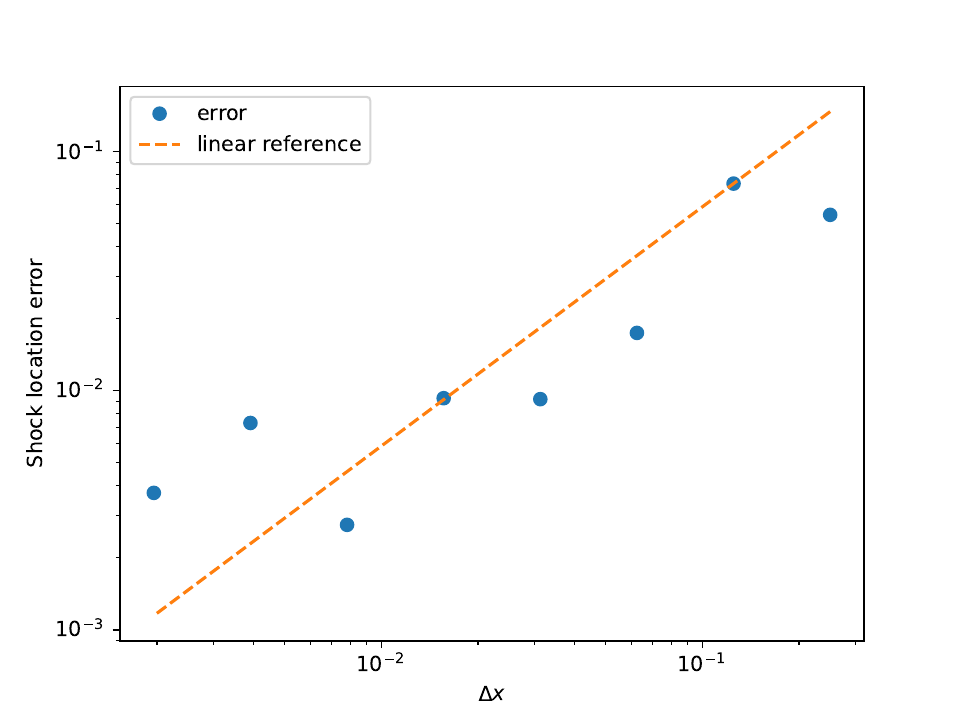}
    \includegraphics[width=0.32\linewidth]{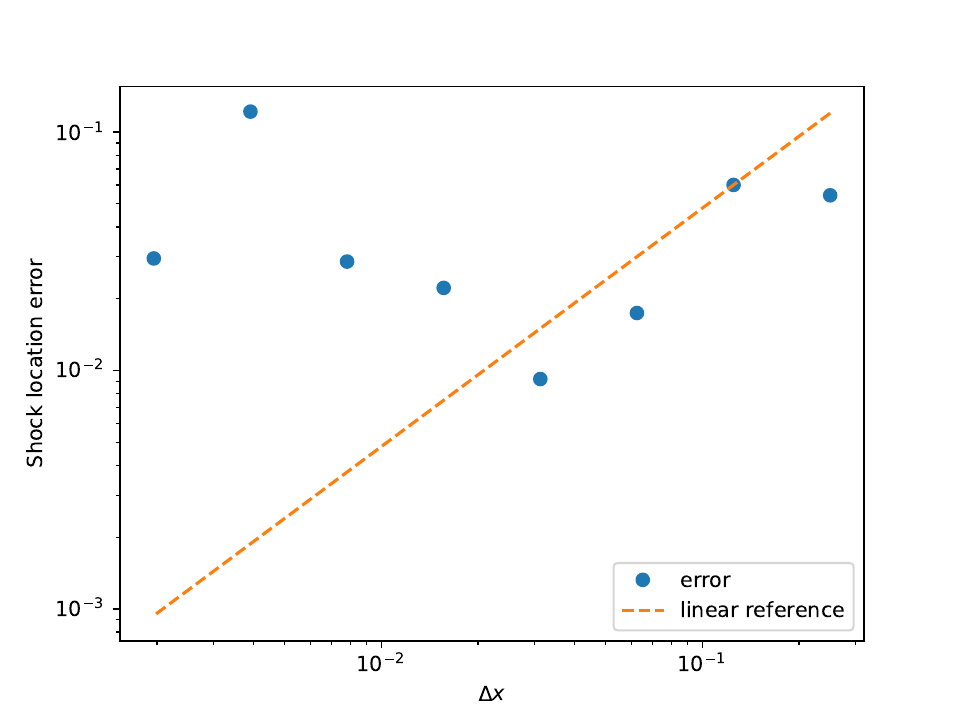}\\
    MPE\\
    \includegraphics[width=0.32\linewidth]{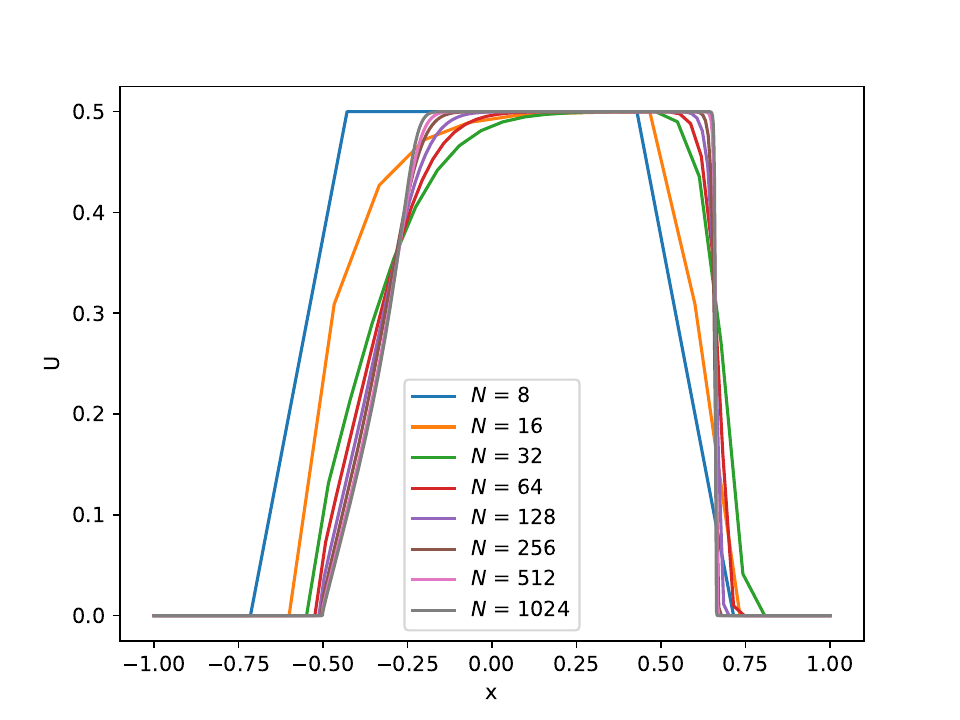}
    \includegraphics[width=0.32\linewidth]{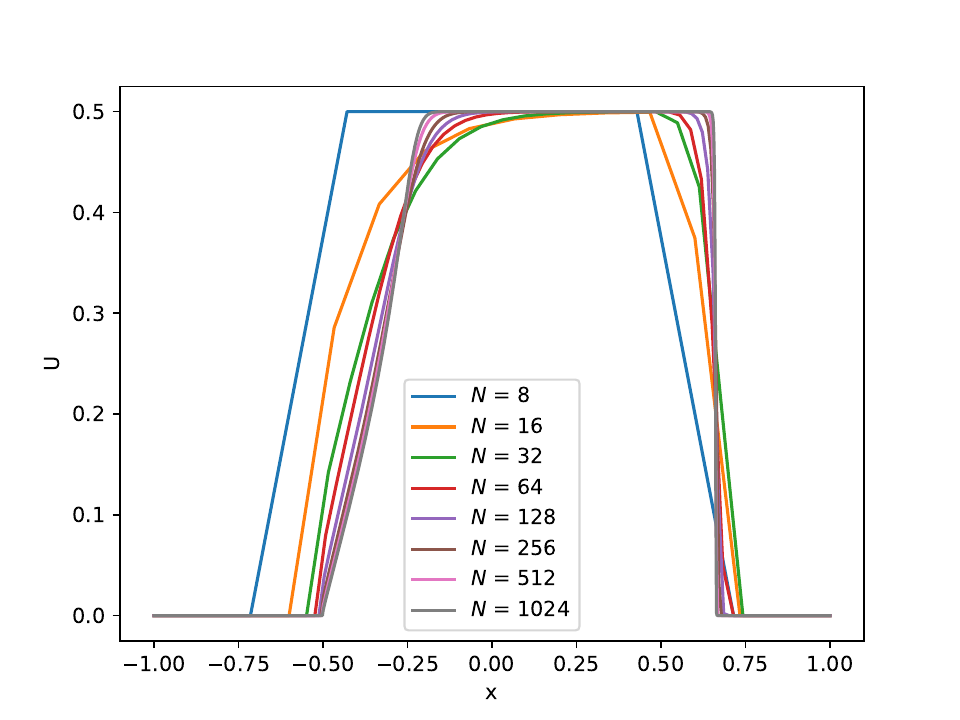}
    \includegraphics[width=0.32\linewidth]{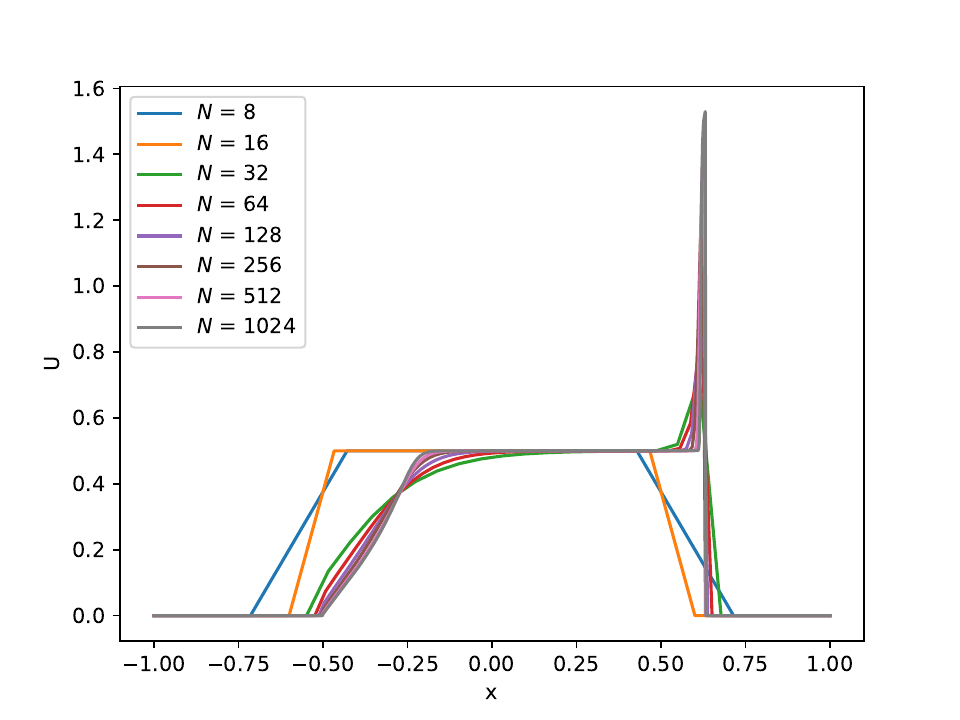}\\
    \includegraphics[width=0.32\linewidth]{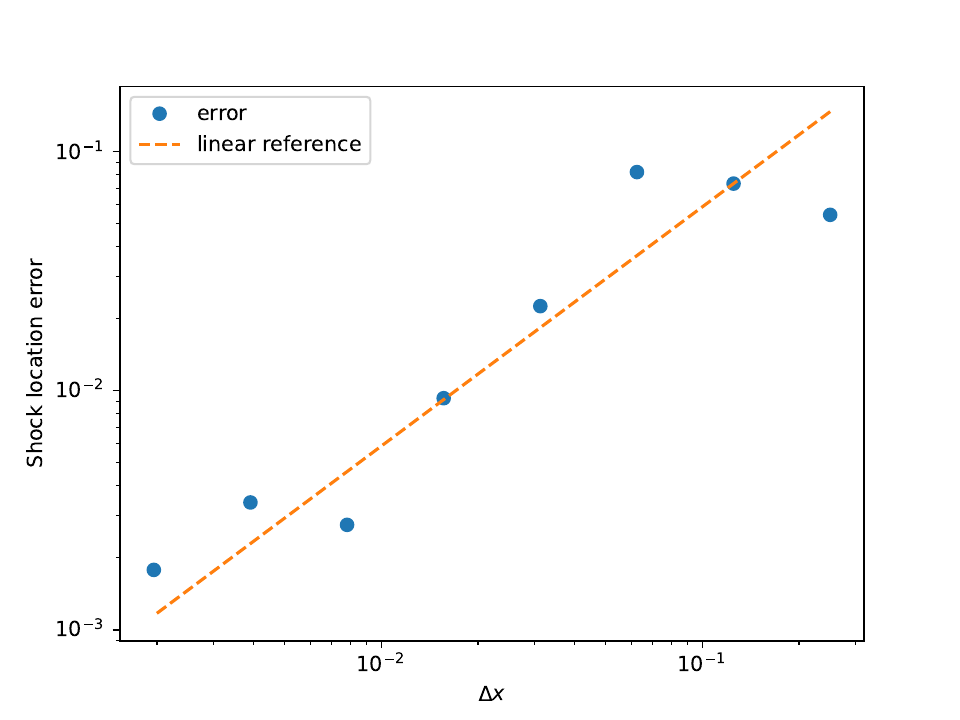}
    \includegraphics[width=0.32\linewidth]{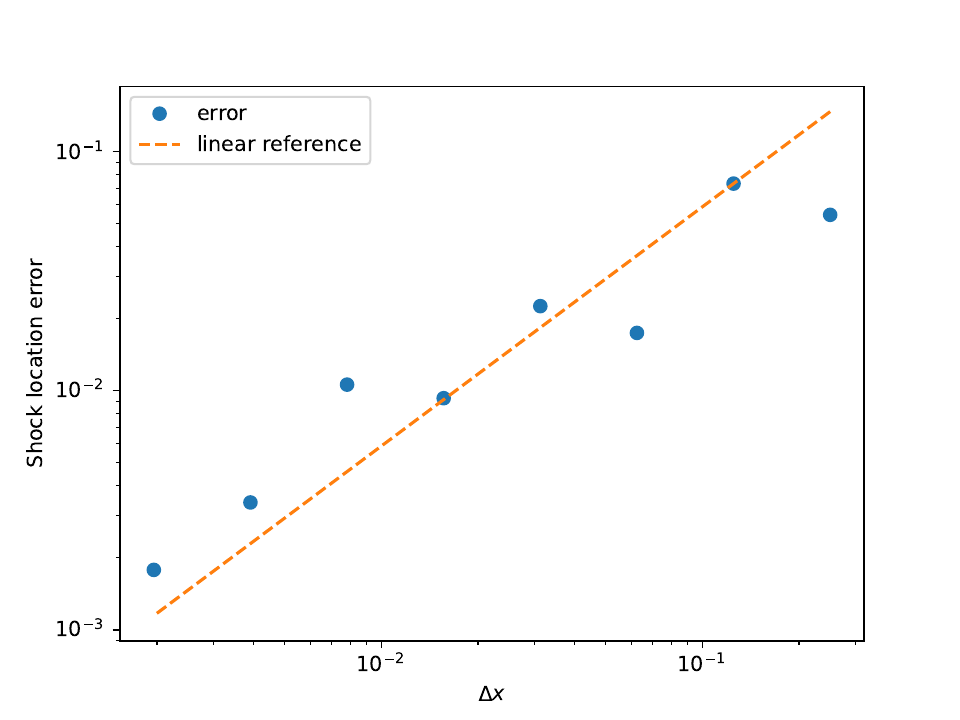}
    \includegraphics[width=0.32\linewidth]{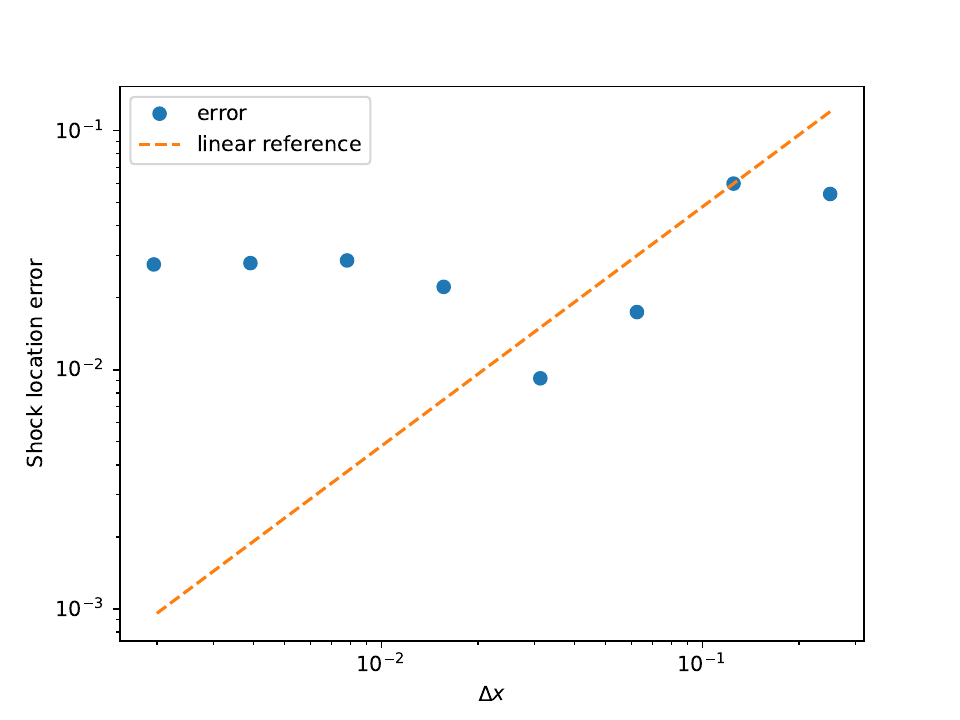}
    \caption{Simulations for explicit Euler (first line), MPE (third line) and numerical shock location errors for explicit Euler (second line) and MPE (last line) for Buckley--Leverett equation \eqref{eq:Buckley} with upwind numerical flux and MPE. $\text{CFL}=0.99$ (left), $\text{CFL}=1.20$ (center) and $\text{CFL}=1.99$ (right)}
    \label{fig:buckley_shock_location}
\end{figure}
Then, we test the Buckley--Leverett equation that approximates a phase concentration in multiphase flows, i.e.,
\begin{equation}\label{eq:Buckley}
    \partial_t u + \partial_x f(u) =0, \qquad \text{with } f(u)= \left( \frac{u^2}{u^2+a(1-u^2)}\right)
\end{equation}
with the IC given by the Riemann problems \eqref{eq:RP}, with $u_1=0.5$ and $u_2=10^{-30}$ and $a=0.5$. Again, the right discontinuity travels with speed $c=\frac{f(u_1)-f(u_2)}{u_1-u_2}$, while a rarefaction fan spreads from the left discontinuity.

In this case, we could prove that the scheme is TVD and, from the numerical simulations, we guess that the MPE is TVD for $\text{CFL}\leq 1$, as well as the explicit Euler method.

In Figure~\ref{fig:buckley_shock_location}, we show the simulations of explicit Euler and MPE at various $N$ for $\text{CFL}=0.99$ (left), $\text{CFL}=1.2$ (center) and $\text{CFL}=1.99$ (right). Below each figure, we also plot the error of the numerical shock location. 
MPE clearly has some qualitative improvements with respect to explicit Euler, in particular for $\text{CFL}=1.2$ where we are outside the stability region for the explicit method, but we can still obtain reliable simulations for MPE. In the explicit method, we observe not only oscillations close to the shock, but also a wrong weak solution at the rarefaction wave (not the entropy solution), while the MPE gives correct results at both waves.
Numerically, we have seen that MPE seems to be TVD up to $\text{CFL}=1$, while explicit Euler is not TVD  for CFL larger than 1. 
In any case, in the regime where explicit Euler converges to the exact solution, we also have that the MPE converges to the exact solution.

\subsection{High order tests}
In this section, we test with higher order schemes the findings of this work. We will first have a high order in time and upwind first order space reconstruction and then we will move to high order in space and time using a WENO spatial reconstruction.

\subsubsection{High order in time}
To test the high order in time Patankar-type schemes, we will use the MPDeC with different order of accuracy, while still using an upwind first order discretization. 
\begin{figure}
    \centering
    \begin{minipage}{0.32\textwidth}
    \centering
        MPDeC(2)
    \end{minipage}\begin{minipage}{0.32\textwidth}
    \centering
        MPDeC(3)
    \end{minipage}\begin{minipage}{0.32\textwidth}
    \centering
        MPDeC(4)
    \end{minipage}\\
    CFL = 1\\
    \includegraphics[width=0.32\linewidth]{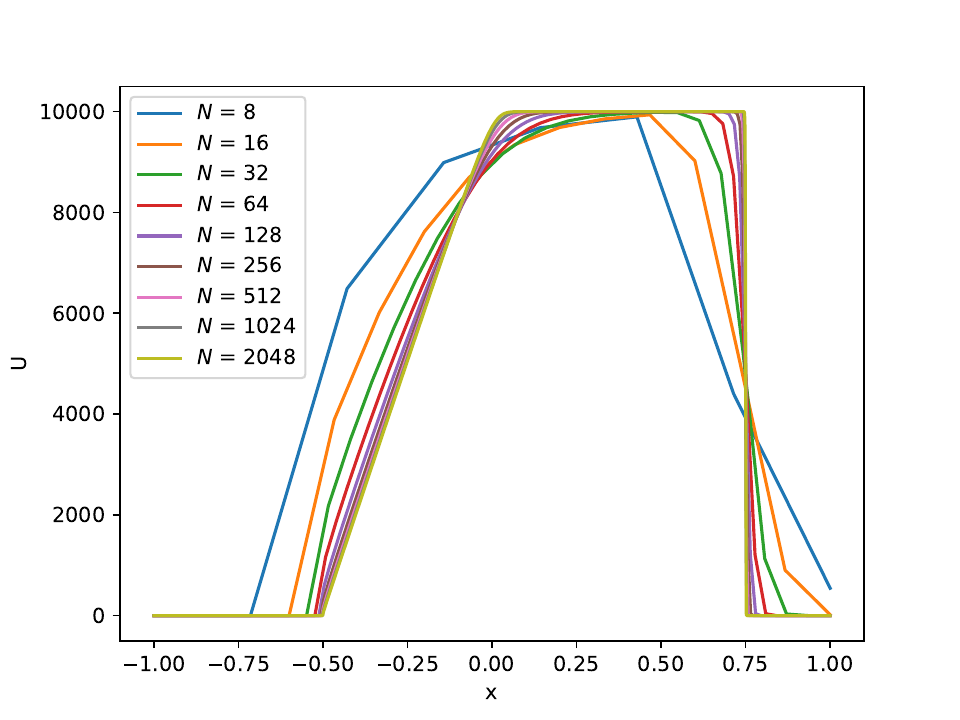}
    \includegraphics[width=0.32\linewidth]{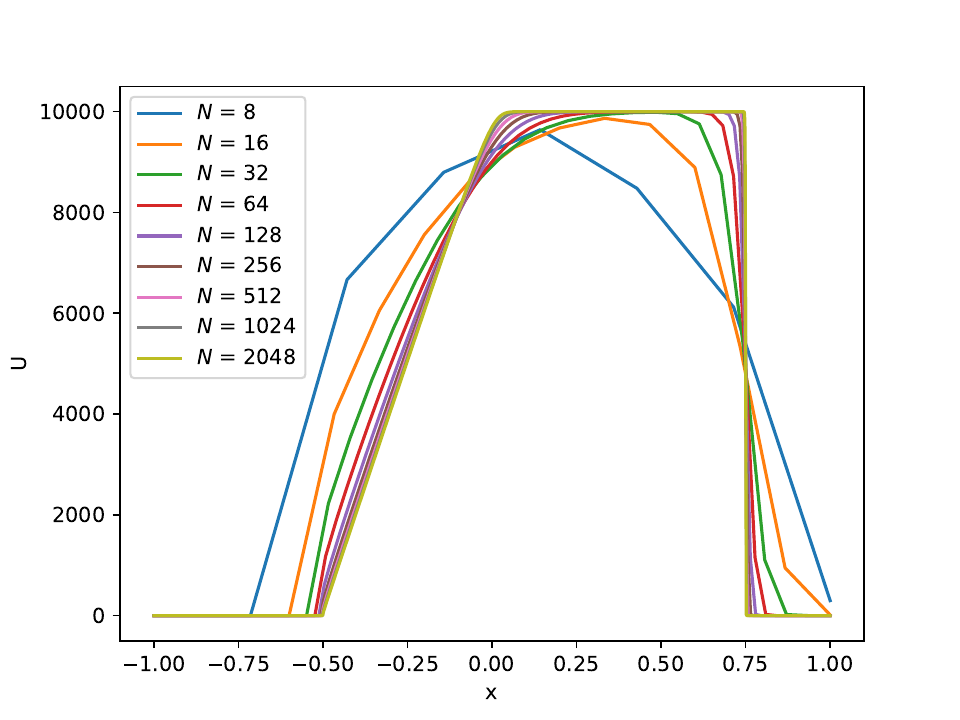}
    \includegraphics[width=0.32\linewidth]{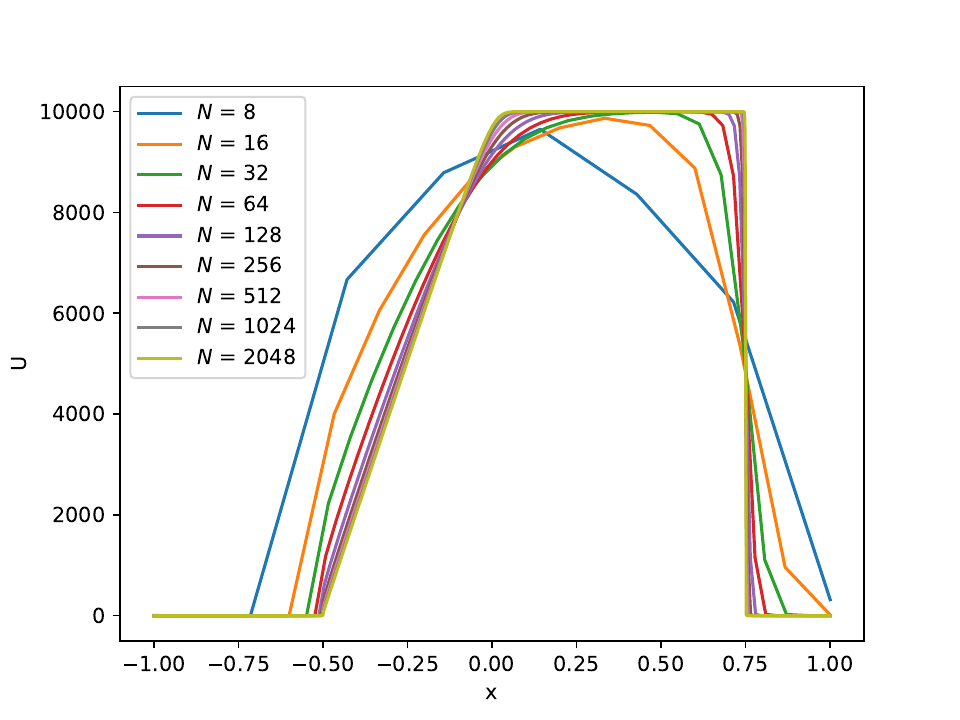}\\
    \includegraphics[width=0.32\linewidth]{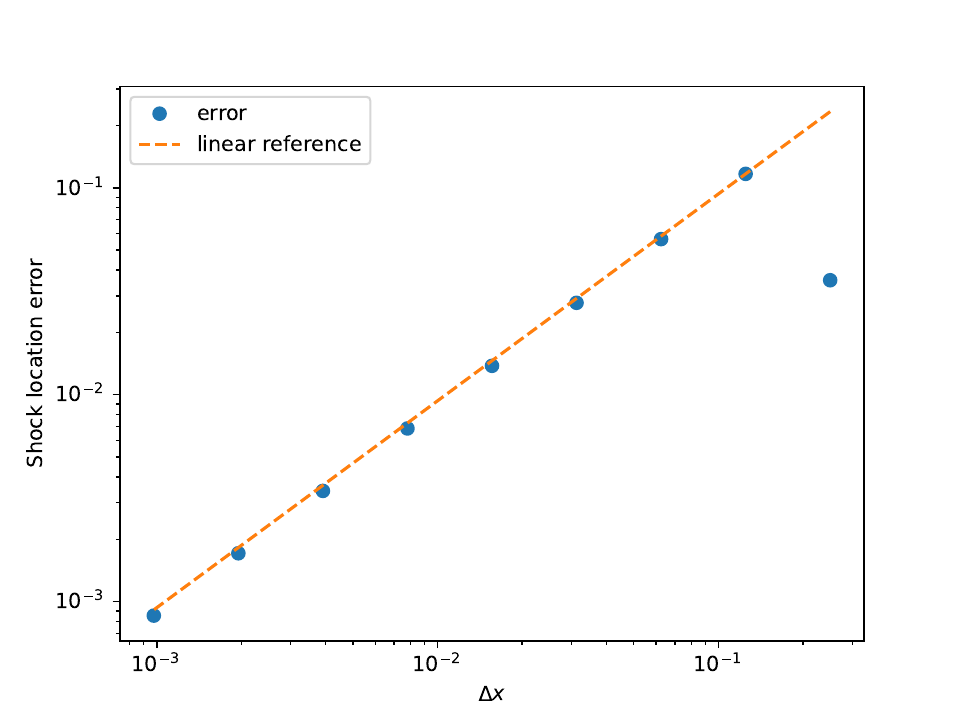}
    \includegraphics[width=0.32\linewidth]{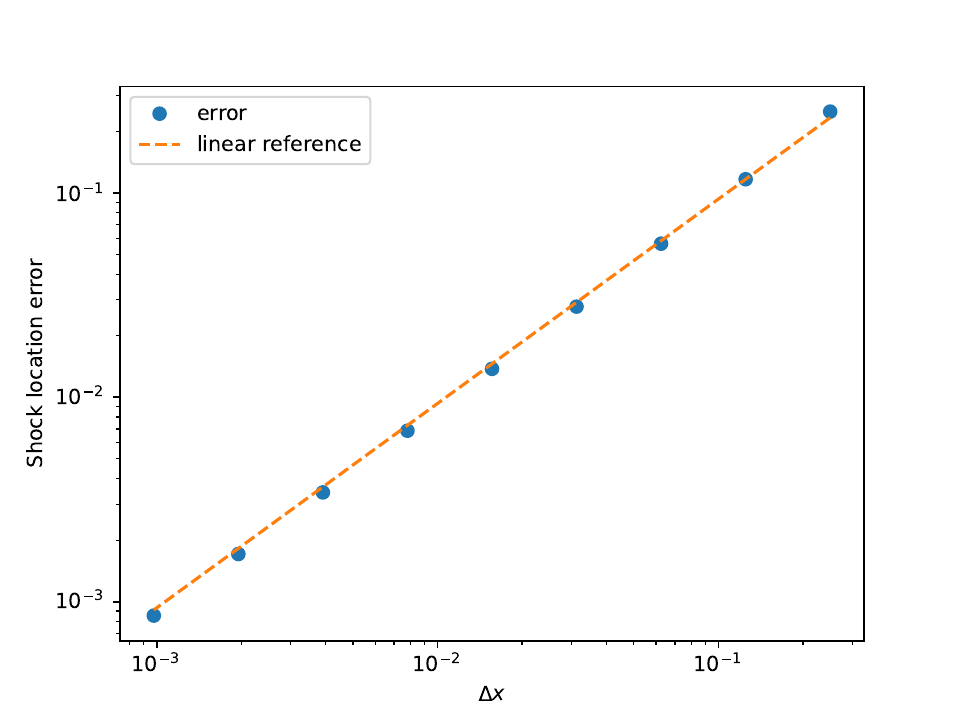}
    \includegraphics[width=0.32\linewidth]{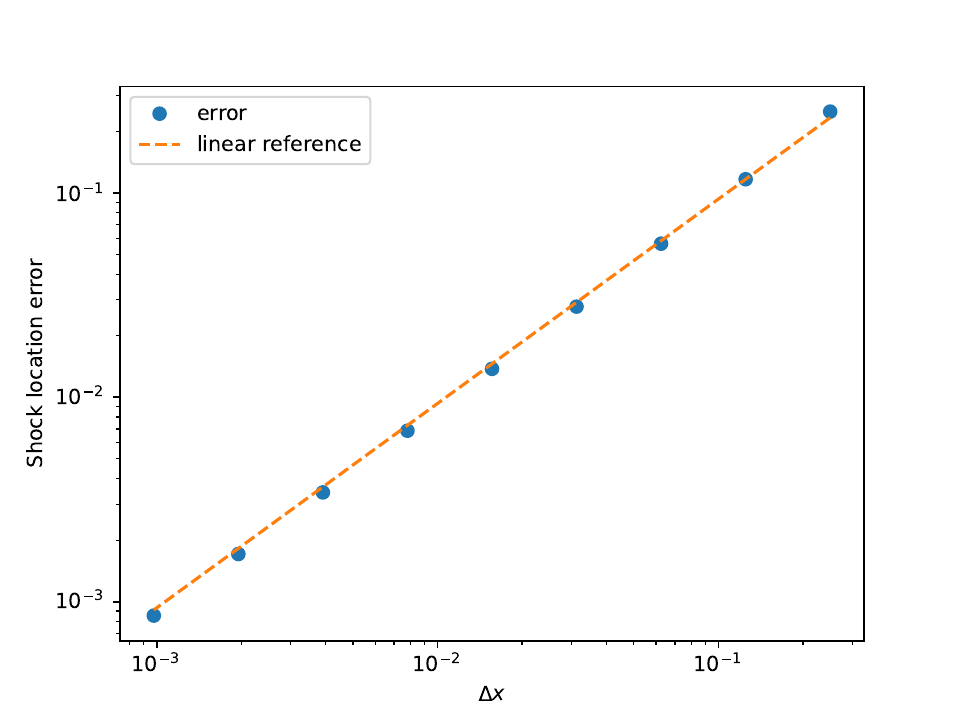}\\
    CFL = 2.1\\
    \includegraphics[width=0.32\linewidth]{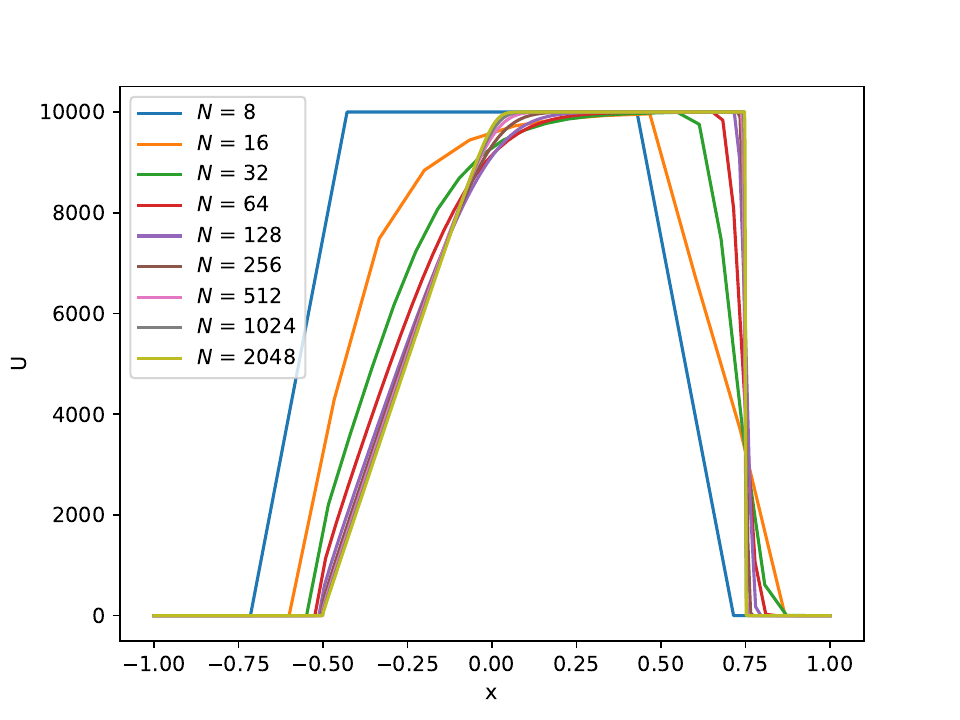}
    \includegraphics[width=0.32\linewidth]{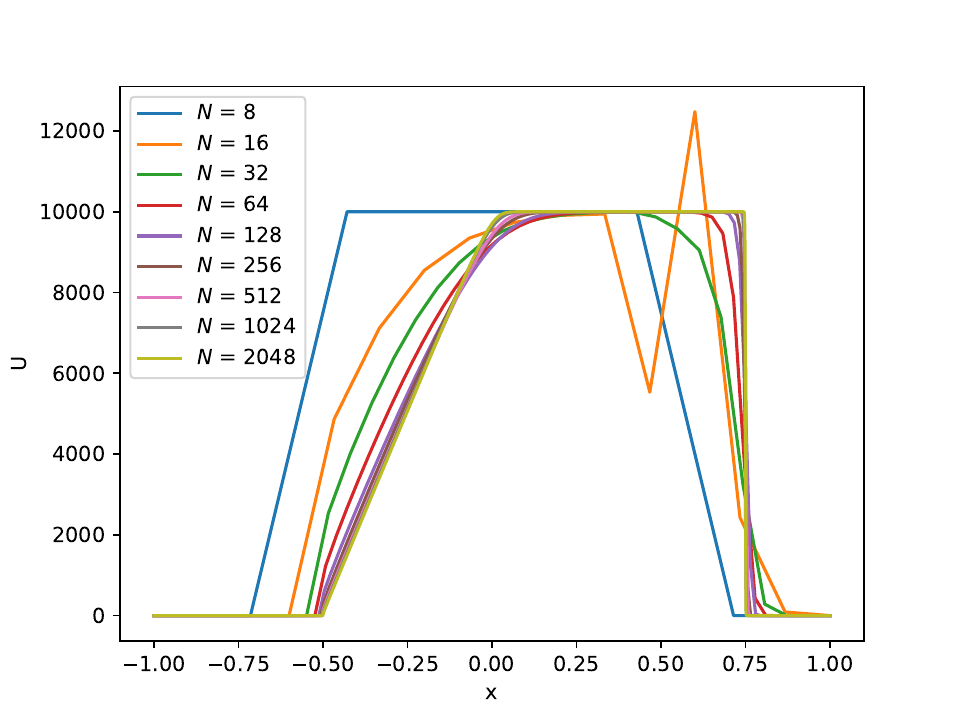}
    \includegraphics[width=0.32\linewidth]{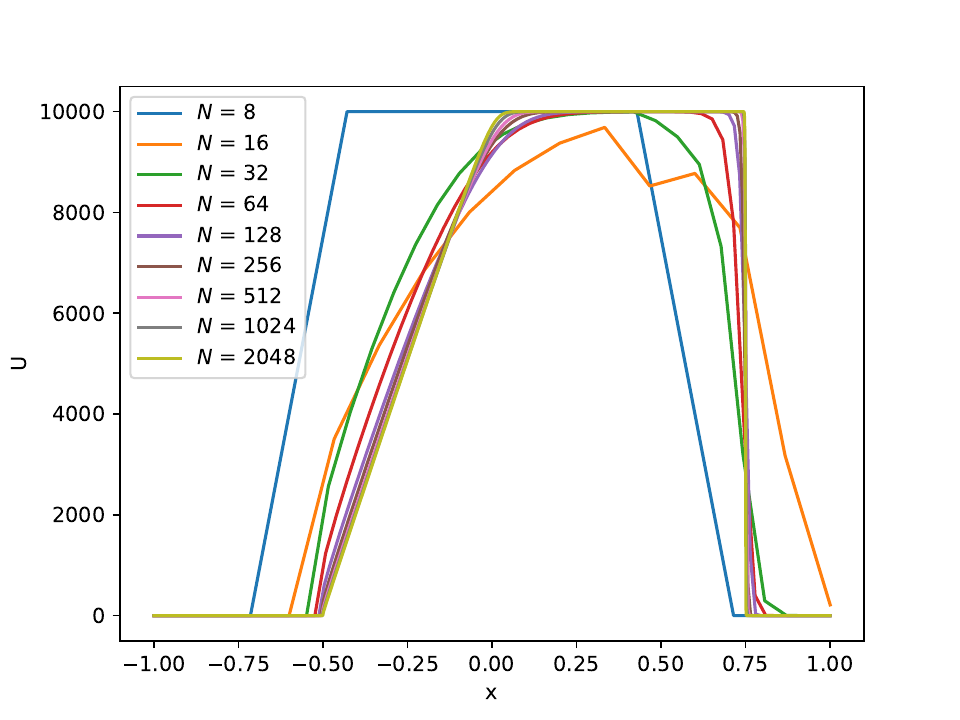}\\
    \includegraphics[width=0.32\linewidth]{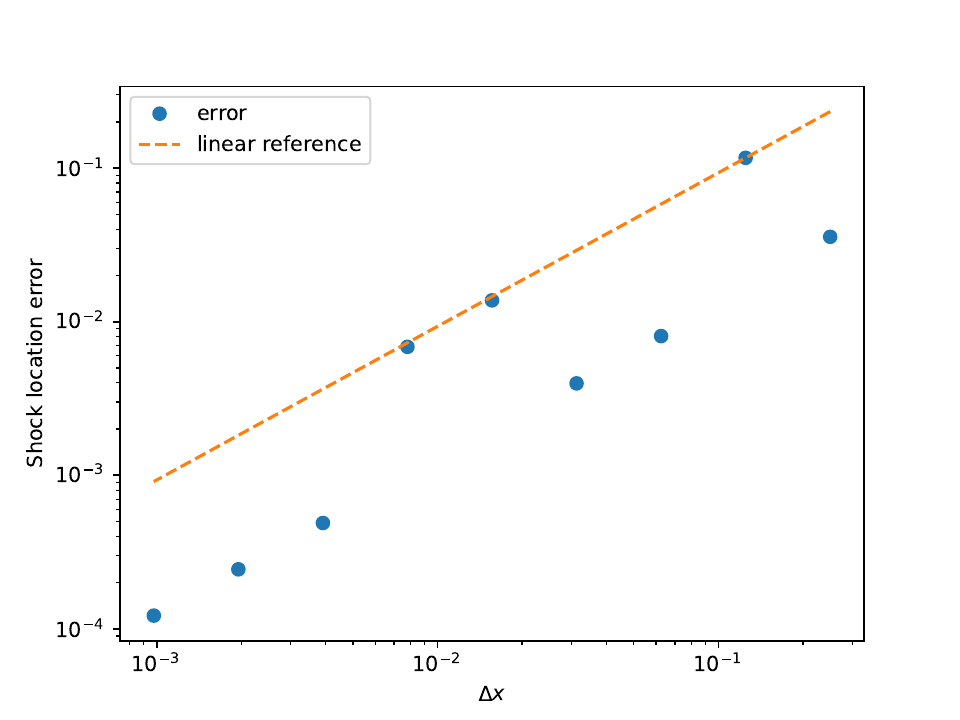}
    \includegraphics[width=0.32\linewidth]{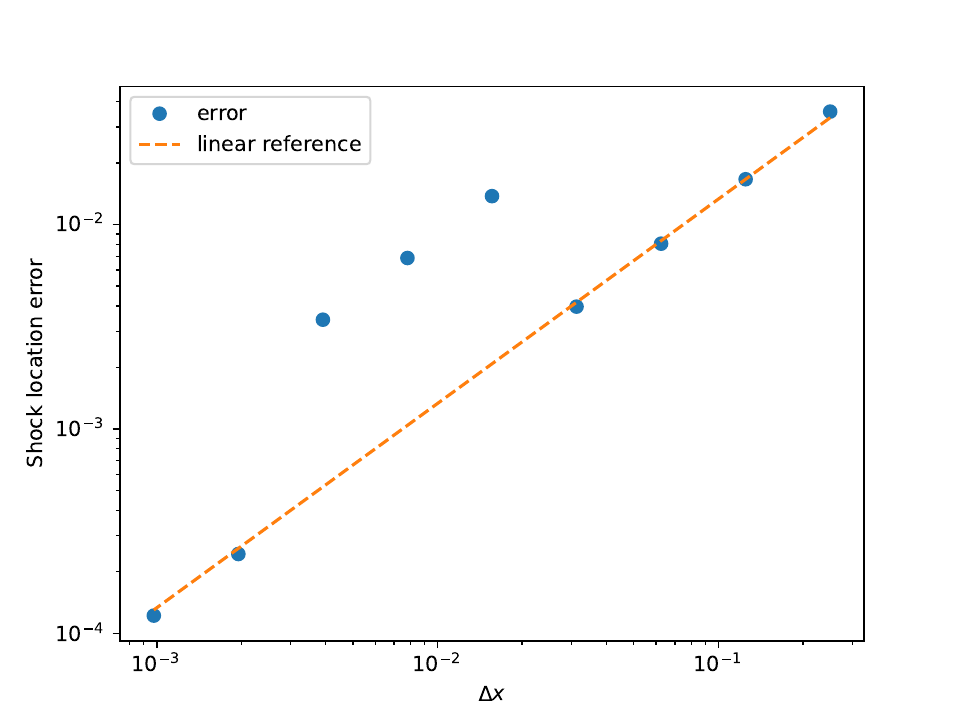}
    \includegraphics[width=0.32\linewidth]{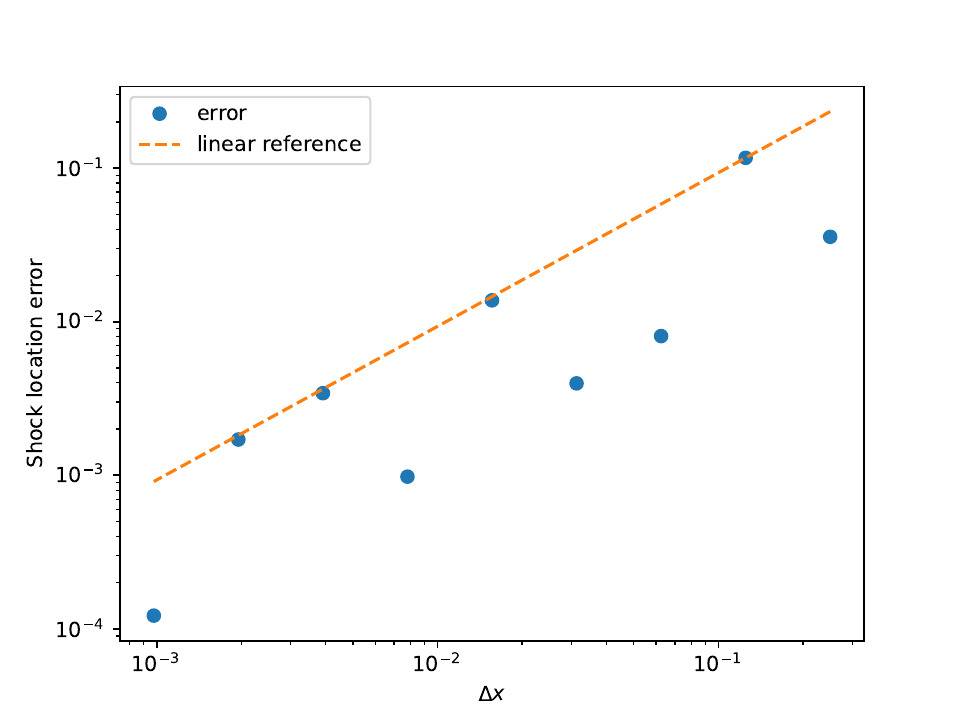}
    \caption{Burgers' equation \eqref{eq:Burgers} simulations and error of the shock location with MPDeC of order 2 (left), 3 (center) and 4 (right) and upwind numerical flux with different CFLs}
    \label{fig:burgers_mpdec_upwind}
\end{figure}
In Figure~\ref{fig:burgers_mpdec_upwind}, we show the simulations for the same double Riemann problem studied in the previous section. 
As for the MPE, we notice a good convergence towards the exact solution, in particular in the shock location at an $O(\Delta x)$ rate. 
For $\text{CFL}=2.1$ there is still convergence, even if the shock location does not converge monotonically to the exact one. It seems that there are two first-order error lines on which the error bounces.

Moreover, numerically, we can observe that also the second order MPDeC(2) for Burgers with upwind seems TVD up to CFL 2, while higher order MPDeC are never TVD even for very small CFLs.

\begin{figure}
    \centering
    \includegraphics[width=\linewidth]{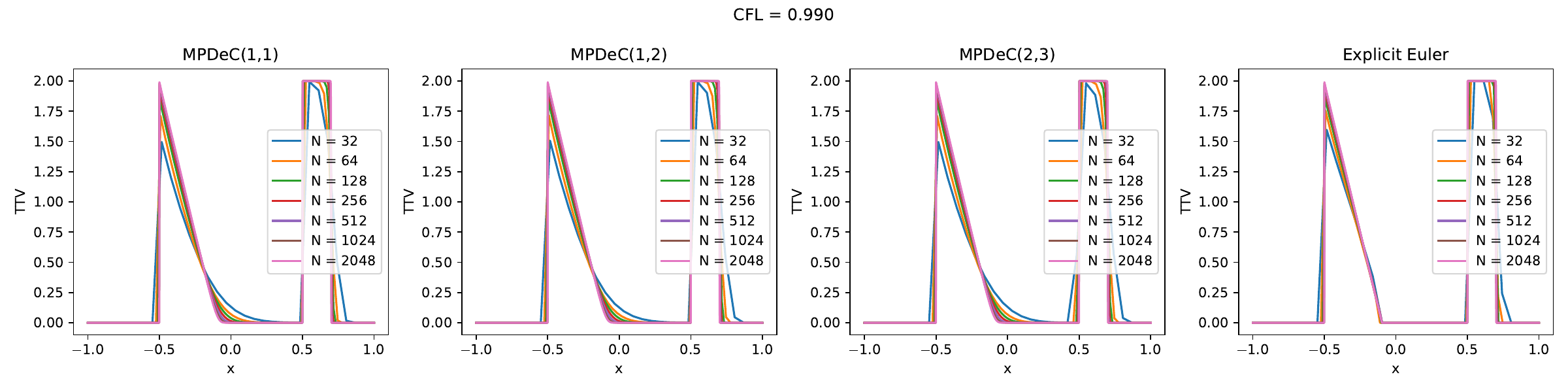}
    \includegraphics[width=\linewidth]{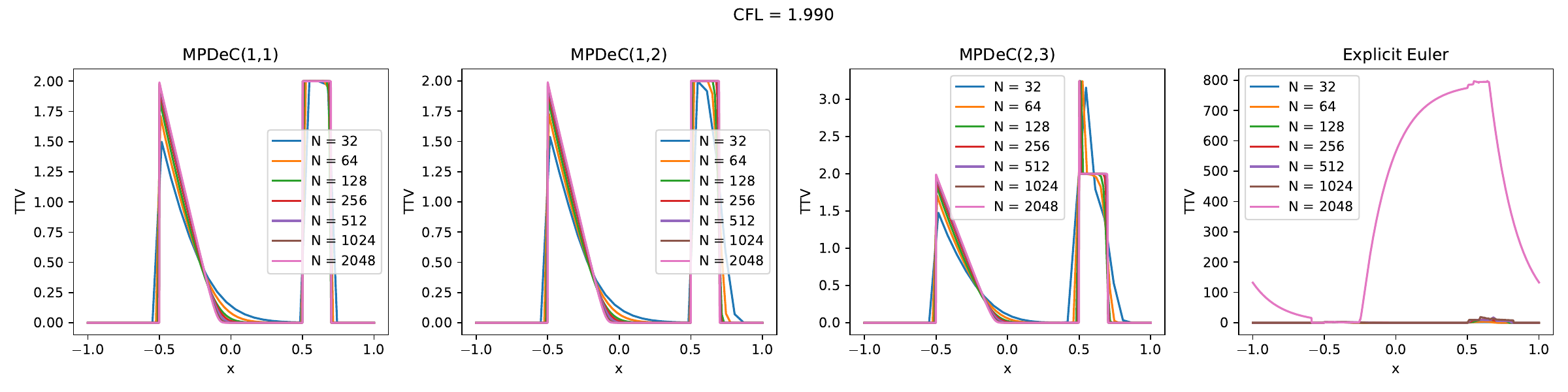}
    \includegraphics[width=\linewidth]{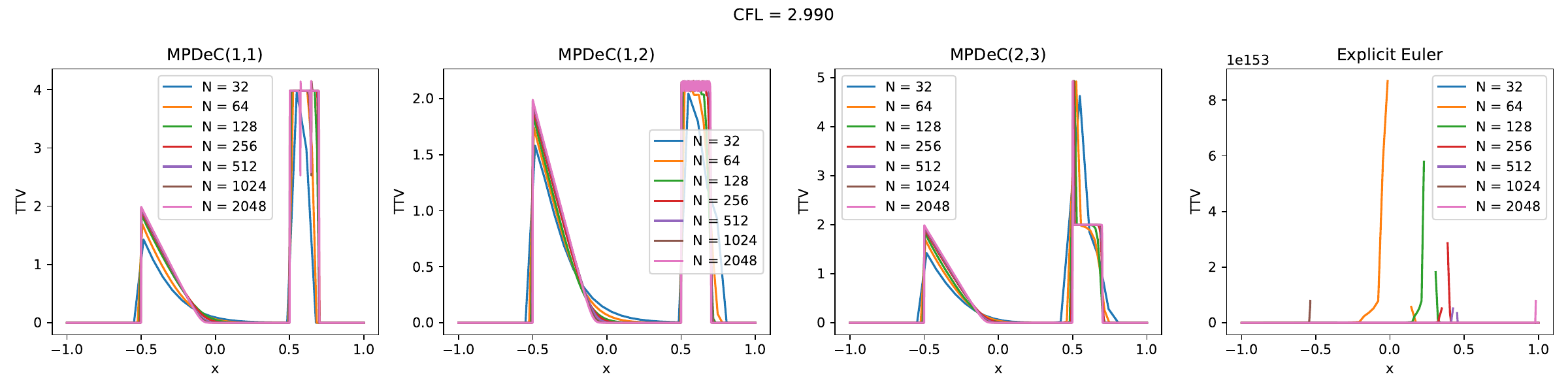}
    \includegraphics[width=\linewidth]{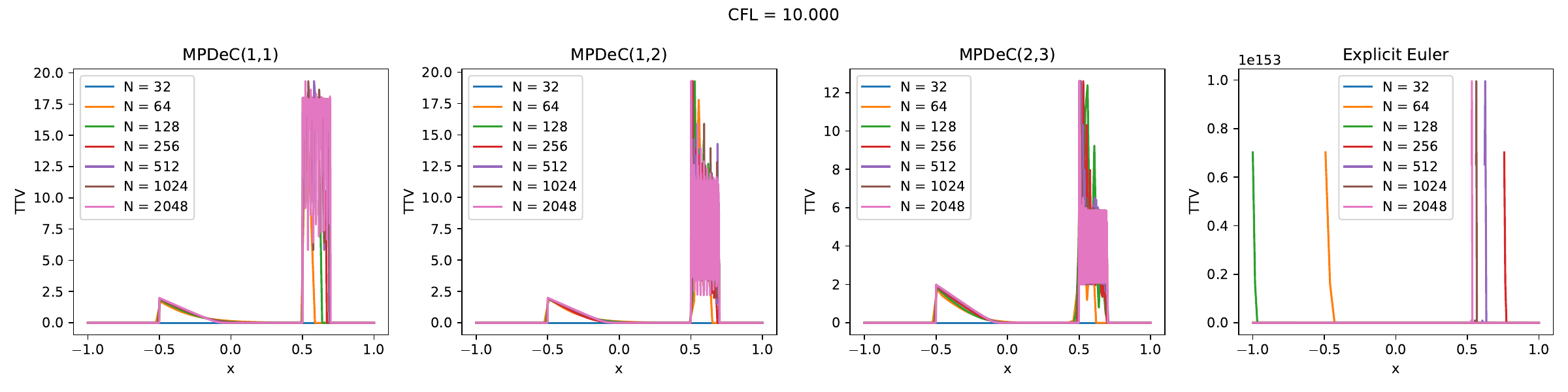}
    \caption{Total time variation for various methods and different CFL number for Burgers' equation \eqref{eq:Burgers} with IC \eqref{eq:RP} with $u_1=2$ and $u_2=10^{-13}$. Observe that for different methods different scales are used.}
    \label{fig:TTV}
\end{figure}

In figure~\eqref{fig:TTV}, we check the TTV boundedness for different time discretization methods at different CFL number (0.99, 1.99, 2.99, 10) for a double RP \eqref{eq:RP} with $u_1=2$ and $u_2=10^{-13}$. The TTV boundedness is the hypothesis that we added to the Lax--Wendroff theorem in order to be proven for Patankar-type methods. Here, we check whether this hypothesis is reasonable for such methods.
In particular, we observe that the MP methods seem to have bounded TTV also for very large CFLs, even if the bound is larger than the exact solution one (which is bounded by 2). This does not prove that the method is uniformly bounded for large CFLs, but it gives good hopes, while for low CFLs it seems that the TTV is bounded by the exact TTV.
On the other side, the explicit Euler method is TTV only for $\text{CFL}\leq 1$, while for larger CFLs has an exploding behavior.

\subsubsection{High order in space and time with WENO-FV}

\begin{figure}
\begin{center}    Burgers' equations \end{center}
	\includegraphics[width=0.295\textwidth, trim={0 0 100 0}, clip]{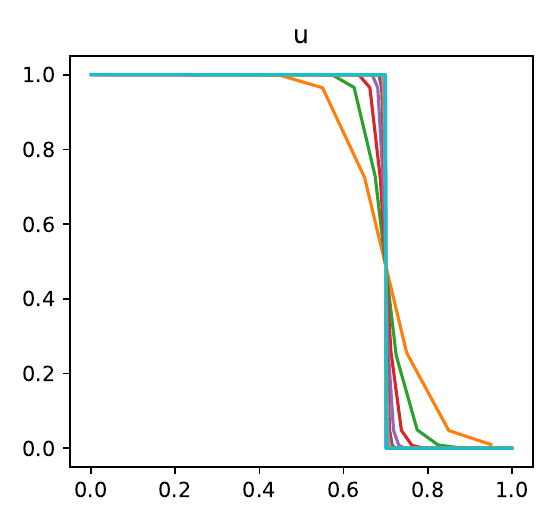}
	\includegraphics[width=0.295\textwidth, trim={0 0 100 0}, clip]{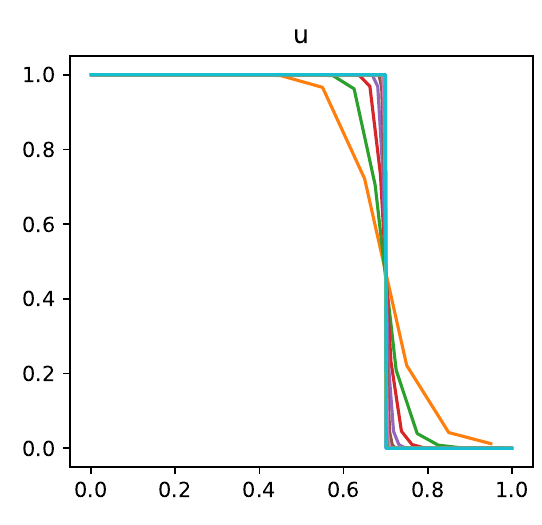}
	\includegraphics[width=0.295\textwidth, trim={0 0 100 0}, clip]{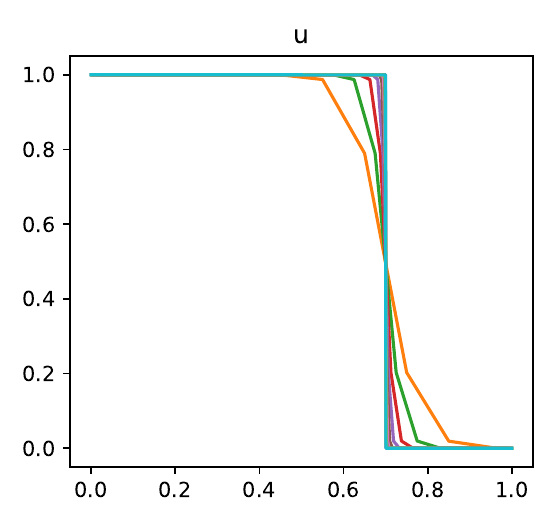}
	\includegraphics[width=0.095\textwidth, trim={268 30 10 0},clip]{Burgers_MPDeC5_comparison_exact_CFL1.pdf}\\
	\includegraphics[width=0.295\textwidth, trim={0 0 0 0}, clip]{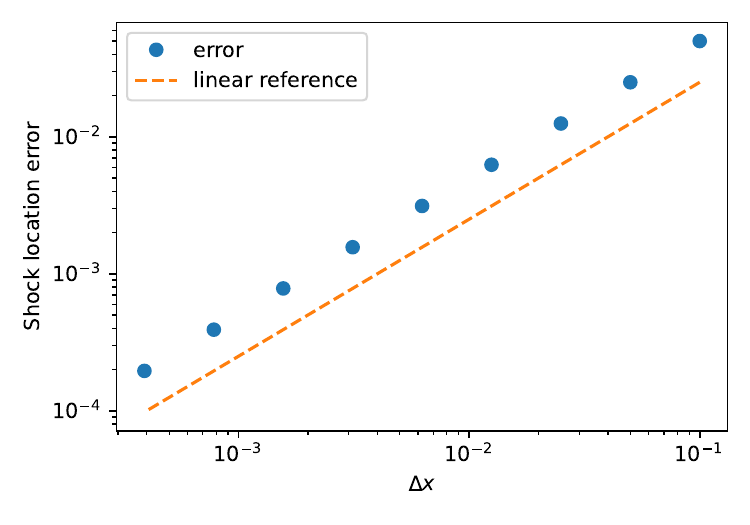}
	\includegraphics[width=0.295\textwidth, trim={0 0 0 0}, clip]{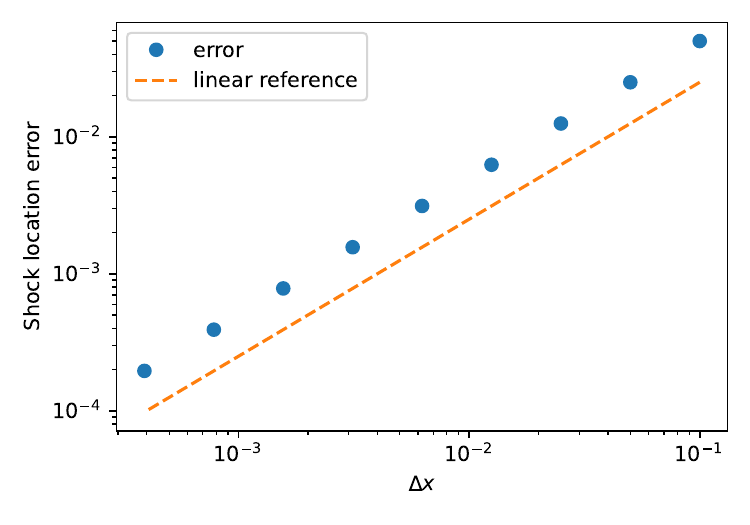}
	\includegraphics[width=0.295\textwidth, trim={0 0 0 0}, clip]{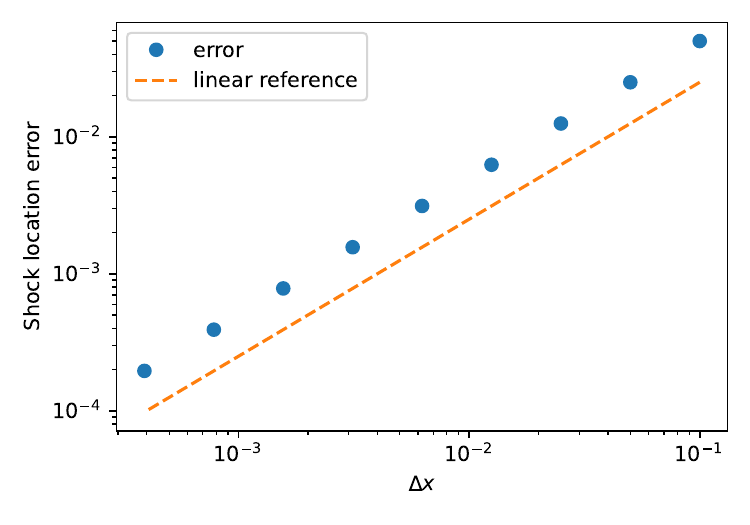}\hfill\\
\begin{center}    Shallow water equations \end{center}
	\includegraphics[width=0.295\textwidth, trim={0 0 440 0},clip]{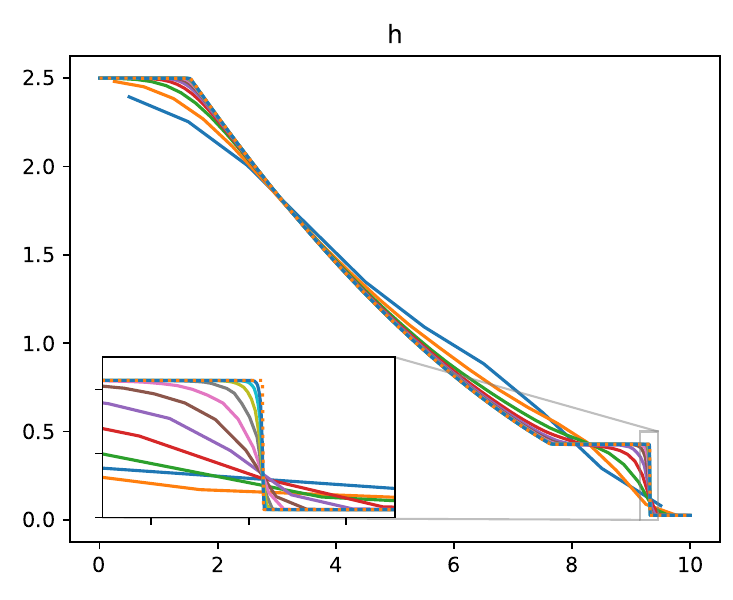}
	\includegraphics[width=0.295\textwidth, trim={0 0 440 0},clip]{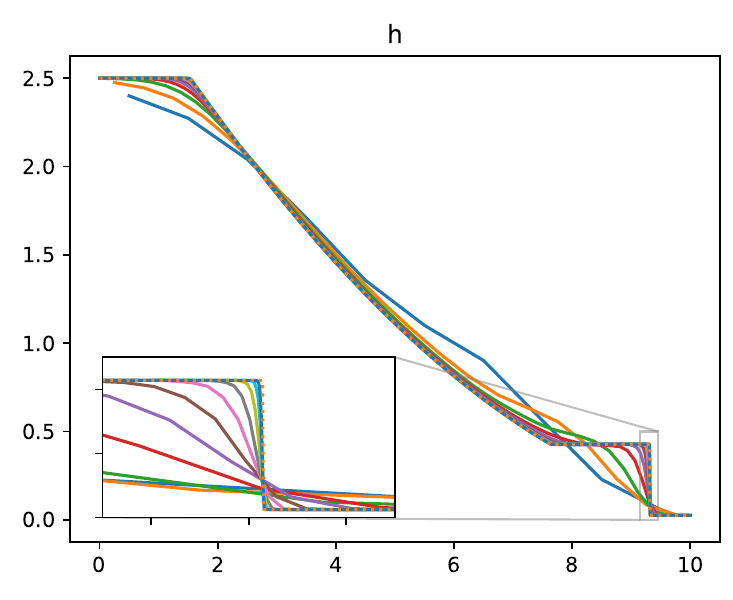}
	\includegraphics[width=0.295\textwidth, trim={0 0 440 0},clip]{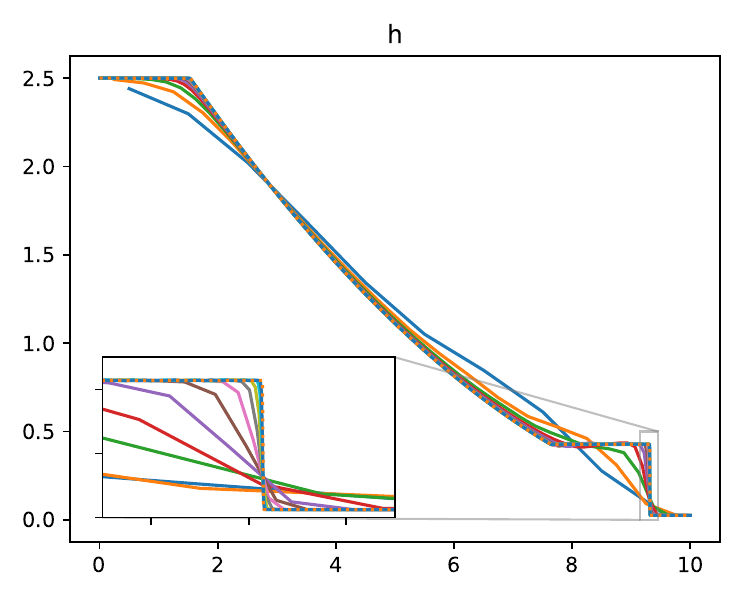}
	\includegraphics[width=0.095\textwidth, trim={690 40 0 0},clip]{SW_MPDeC5_comparison_exact_CFL1.pdf}\\
    \includegraphics[width=0.295\textwidth, trim={0 0 0 0}, clip]{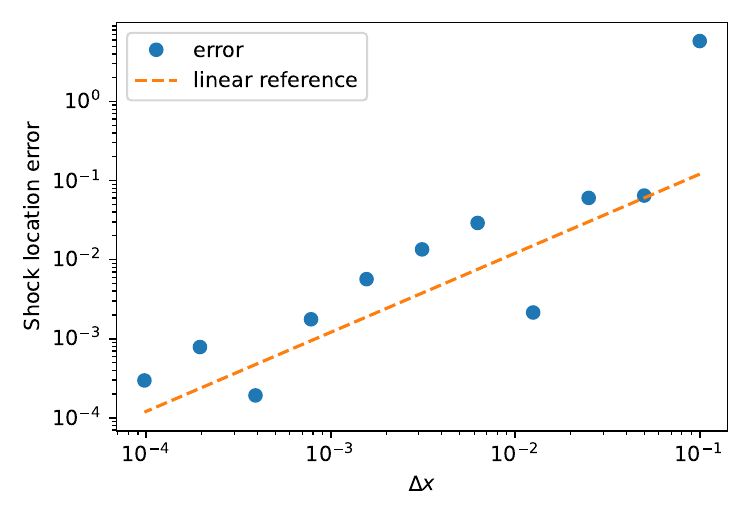}
	\includegraphics[width=0.295\textwidth, trim={0 0 0 0}, clip]{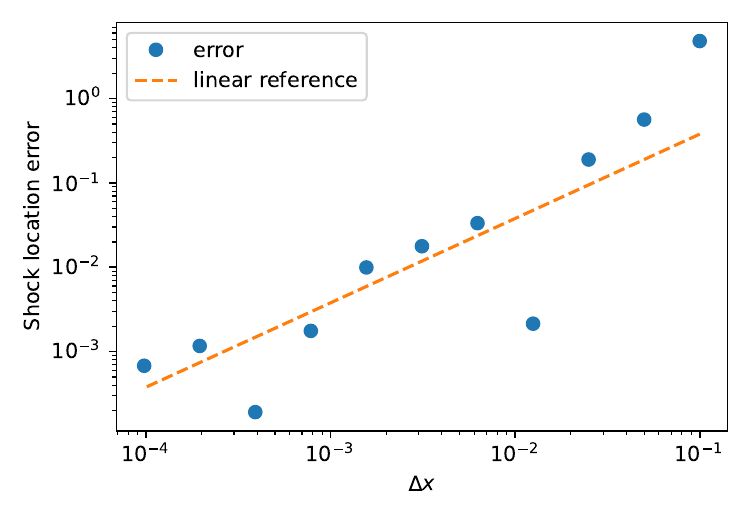}
	\includegraphics[width=0.295\textwidth, trim={0 0 0 0}, clip]{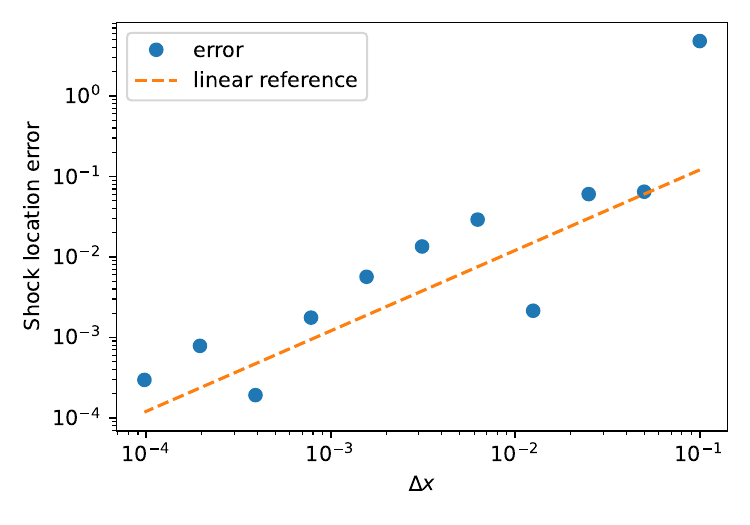}\hfill\\
	\caption{Burgers' (top) and SW dam break (bottom), solutions and shock location error refining the mesh: MPDeC2 (left),  MPSSPRK3 (center) and MPDeC5 (right)}\label{fig:tests}
\end{figure}

Finally, in this section we propose a space-time high order discretization which involves a FV WENO spatial reconstruction \cite{shu_efficient_1988} in each cell, with a positivity limiter \cite{perthame1996positivity} to guarantee the positivity in each reconstruction point. Typically these limiters are set to work with the explicit Euler method and SSPRK methods, where the CFL has to be adjusted according to the weights of the underlying quadrature, e.g. for order 5 one needs a CFL of $\frac{1}{12}$.
In the context of Patankar-type schemes, it has been shown that this restriction of the CFL is not required \cite{ciallella2022arbitrary}.

Here, we use the discretization proposed in \cite{ciallella2022arbitrary} where the FV-WENO discretization order matches the time integration order of accuracy.
As time discretization, we use various Patankar-type methods, i.e.\ MPDeC \cite{offner2020arbitrary} of various order and the optimal third order MP strong-stability-preserving Runge--Kutta method \cite{huang2019third} (MPSSPRK3).

Also here, we propose two tests for different conservation laws with discontinuities with a state close to 0, for Burgers' equations and shallow water (SW) equations. Differently from the previous test cases, in this code a minimum quantity for the positive variable is set at $10^{-14}$ to avoid blow-ups in the velocity computation (in SW). 

We start with the inviscid Burgers' equations $\partial_t u(x,t) +\partial_x (u^2(x,t)/2) =0$ for $u:[0,1]\times \mathbb R^+\to \mathbb R$  with IC and exact solution
\begin{equation*}
	u(x,0)=\begin{cases}
		1& x\leq 0.2,\\
		10^{-16}&\text{else,}
	\end{cases}\qquad 
	u(x,t) = \begin{cases}
		1& x\leq 0.2+\frac12 t,\\
		10^{-16}& \text{else}.
	\end{cases}
\end{equation*}
The solution of this problem is clearly positive, hence it is reasonable to apply the MP procedure on the variable $u$. In Fig.~\ref{fig:tests} (top), we show the simulations for various methods at time $T=1$ obtained with $\text{CFL}=0.99$. In all schemes, we see a convergence of the shock towards the exact solution, in accordance with Remark~\ref{rmk:RK} and, hence, with the shock traveling at the right speed. This is a symptom of a conservative method and consistent numerical fluxes.

We then proceed with a shallow water equation test of a slightly wet dam break 
\begin{equation}\label{eq:dam_break}
	\begin{cases}
		\partial_t h + \partial_x (hu) =0,\\
		\partial_t (hu) + \partial_x (hu^2+\frac{9.8}{2}h^2) =0,
	\end{cases}
	\qquad
	h(x,0)=\begin{cases}
		2.5&x\leq 5,\\
		0.025&x> 5,
	\end{cases}
	\quad u(x,0)=0.
\end{equation}
Clearly, for this test only $h$ is part of a positive production-destruction system and is put in an MP solver, while $hu$ is solved with the underlying explicit method.
The choice of a wet dam break is dictated by the fact that dry dam breaks do not show a discontinuity close to the dry area, but rather a rarefaction wave, making it impossible to validate the speed of the shock close to the dry region.
Moreover, we remark that, in particular in the initial times, the MP weights close to the shock are proportional to the ratio of the left and the right state of $h$, which in this case is 100, see \eqref{eq:dam_break}. 
We believe that this jump would already influence the behavior of the solution. Larger jumps would make the shock smaller and the numerical diffusion would smear out their behavior, making the validation of the result impossible.

In Fig.~\ref{fig:tests} (bottom), we  depict the simulation at time $T=0.7$ obtained with $\text{CFL} =0.99$ for different MP solvers next to the exact solutions as well as the numerical shock location error. Here, we also observe that the solutions converge to the exact one as the mesh is refined.

In~\cite{ciallella2022arbitrary,ciallella2025high}, other more challenging shallow water tests have been solved using the same approach with MPDeC. There, multiple shocks and their interactions in wet/dry regimes have been tested and in all those simulations the solutions showed reliable results in comparison with classical 
benchmarks.

\section{Conclusions and perspectives}\label{se_conclusion} 
We have extended the Lax-Wendroff theorem to Patankar-type temporal discretization of difference methods for conservation laws. The proof of the theorem required an additional hypothesis, with respect to the classical one, that bounds the total time variation of the numerical solutions uniformly in space and refinements. 
The numerical results show that, if convergent, the Patankar-type methods converge to the same weak solution of the explicit methods, and sometimes they converge also when the explicit methods fail at converging, bringing in some extra stabilization.

It would be interesting to extend this result, for example, proving that the extra hypothesis is not needed. In this context, ideas similar to \cite{birken_conservation_2022} could be explored, even if there is no guarantee that the two limits of linear solver iterations going to infinity and mesh refinements can be exchanged. 
Some of the authors are also working on defining an entropy-preserving version of Patankar-type methods. This will be of particular interest for the development of Patankar-type methods that guarantee the positivity of quantities that are not directly evolved by the conservation law but instead are derived by an additional equation of state. For instance, the pressure in the context of the compressible Euler equations would be positive if the entropy is dissipated along the simulation and the density is positive, see for example \cite{berthon2025towards}. 
This still requires some additional work and it is not a trivial modification of Patankar-type schemes for PDS.

\section*{Acknowledgements}
P.\,Ö. is supported by the DFG within SPP 2410, project  525866748 and under the personal grant 520756621. D.\,T. was supported by the Ateneo Sapienza projects 2023 “Modeling, numerical treatment of hyperbolic equations and optimal control problems”.

\noindent \section*{In memoriam}
 
\noindent This paper is dedicated to the memory of Prof.\ Arturo Hidalgo L\'opez
($^*$July 03\textsuperscript{rd} 1966 - $\dagger$August 26\textsuperscript{th} 2024) of the Universidad Politecnica de Madrid, organizer of HONOM 2019 and active participant in many other editions of HONOM.
Our thoughts and wishes go to his wife Lourdes and his sister Mar\'ia Jes\'us, whom he left behind.

\appendix
\section{Total Variation Diminishing property}\label{sec:TVD}

%
%

In this section, we will prove an interesting result for a very specific combination of space-time discretization on Burgers' equation. 
We use MPE as time discretization, an upwind numerical flux (being $u$ positive, this is quite straightforward) with the definition of PDS of \eqref{eq:prod_dest}. 
We can prove that this scheme is total variation diminishing (TVD) for an \textit{implicit} $\text{CFL} <2$.
\begin{thm}
	Consider the upwind method with the following PDS distribution
	\begin{equation}
		\partial_t U_i = - \frac{U_i^2-U_{i-1}^2}{2\Delta x}, \qquad \text{with } p_{i,i-1}=\frac{U_{i-1}^2}{2\Delta x}, \quad d_{i,i+1}=\frac{U_{i}^2}{2\Delta x}.
	\end{equation}
	The MPE scheme applied to this semidiscretization is TVD with the \textit{implicit }CFL condition
    \begin{equation}\label{eq:implicit_CFL}
        \frac{\Delta t^n}{\Delta x} U_i^{n+1} \leq 2 \qquad \forall i,
    \end{equation}
    for all times $n$ with $\Delta t ^n=t_{n+1}-t_n$.
\end{thm}
\begin{proof}
	Let us first write down explicitly the MPE scheme applied to the previous semidiscretization:
	\begin{align}
		U^{n+1}_i &= U^n_i - \frac{\Delta t}{2\Delta x}\left( (U_i^n)^2 \frac{U_i^{n+1}}{U_i^n}-(U_{i-1}^n)^2 \frac{U_{i-1}^{n+1}}{U_{i-1}^n} \right)\\
		&=U^n_i - \frac{\Delta t}{2\Delta x}\left( U_i^n U_i^{n+1}- U_{i-1}^n U_{i-1}^{n+1}\right) =U^n_i - \frac{\Delta t}{2\Delta x}\left( U_i^n (U_i^{n+1}-U_{i-1}^{n+1})+ (U^n_i-U_{i-1}^n) U_{i-1}^{n+1}\right) .
	\end{align}
	Now, applying the difference of two consecutive values, we obtain
	\begin{align}
    \begin{split}
	&U^{n+1}_i-U^{n+1}_{i-1} =U^n_i -U^n_{i-1}\\
        &- \frac{\Delta t}{2\Delta x}\left( U_i^n (U_i^{n+1}-U_{i-1}^{n+1})+ (U^n_i-U_{i-1}^n) U_{i-1}^{n+1}  
		-U_{i-1}^n (U_{i-1}^{n+1}-U_{i-2}^{n+1})- (U^n_{i-1}-U_{i-2}^n) U_{i-2}^{n+1}\right)
    \end{split}
	\end{align}
	and bringing on the left hand side the implicit term that we can collect, we obtain
	\begin{align}
    \begin{split}
		&(U^{n+1}_i-U^{n+1}_{i-1})\left(1+\frac{\Delta t}{2\Delta x}U^n_i\right) \\
        =&(U^n_i -U^n_{i-1})\left(1-\frac{\Delta t}{2\Delta x}U^{n+1}_{i-1}\right)+ \frac{\Delta t}{2\Delta x}\left( 
		U_{i-1}^n (U_{i-1}^{n+1}-U_{i-2}^{n+1})+(U^n_{i-1}-U_{i-2}^n) U_{i-2}^{n+1}\right).
    \end{split}
	\end{align}
	Now, using \eqref{eq:implicit_CFL} and the positivity of the solution, 
	we obtain	
	\begin{align}
		\begin{split}
			\sum_i |U^{n+1}_i-U^{n+1}_{i-1}|&\left(1+\frac{\Delta t}{2\Delta x}U^n_i\right) \leq\sum_i|U^n_i -U^n_{i-1}|\left(1-\frac{\Delta t}{2\Delta x}U^{n+1}_{i-1}\right)\\
			&+ \sum_i\frac{\Delta t}{2\Delta x}
			U_{i-1}^n |U_{i-1}^{n+1}-U_{i-2}^{n+1}|+ \sum_i\frac{\Delta t}{2\Delta x}|U^n_{i-1}-U_{i-2}^n|U_{i-2}^{n+1}
		\end{split}\\
		\begin{split}
			\sum_i |U^{n+1}_i-U^{n+1}_{i-1}| &\left(1+\frac{\Delta t}{2\Delta x}U^n_i -\frac{\Delta t}{2\Delta x}
			U_{i}^n \right)  \leq\sum_i|U^n_i -U^n_{i-1}|\left(1-\frac{\Delta t}{2\Delta x}U^{n+1}_{i-1}+ \frac{\Delta t}{2\Delta x}U_{i-1}^{n+1} \right)
		\end{split}\\
		\Longleftrightarrow TV(U^{n+1})&=\sum_i |U^{n+1}_i-U^{n+1}_{i-1}|\leq \sum_i|U^n_i -U^n_{i-1}|=TV(U^{n}).
	\end{align}
\end{proof}
It is not so easy to use the same trick for other schemes and systems as the add-and-subtract trick we did was tailored for Burgers' equation and those Patankar weights. For more complex problems and schemes one would have further terms that are not so easy to handle.
\begin{rmk}[On the implicit CFL condition]
    The implicit CFL condition \eqref{eq:implicit_CFL} is not the classical CFL condition imposed using the values of the solution at the previous time step, but it rather implies knowing the solution at the next time step, which is not possible. Nevertheless, from an experimental point of view, we have observed that also imposing the classical explicit CFL condition with CFL value equal to 2 is sharply guaranteeing the TVD property.
\end{rmk}
\subsection{Numerical results}
To test the TVD property, we run some simulations for the Burgers' equation with a double Riemann problem \eqref{eq:RP} as initial condition with different values for $u_1$ and $u_2$. The numerical findings are independent on these values. We show the simulations for $u_1=2$ an $u_2=10^{-13}$ with $N=100$ discretization points and periodic boundary conditions.
\begin{figure}
    \centering
    \begin{minipage}{0.49\textwidth}
        \includegraphics[width=0.99\linewidth]{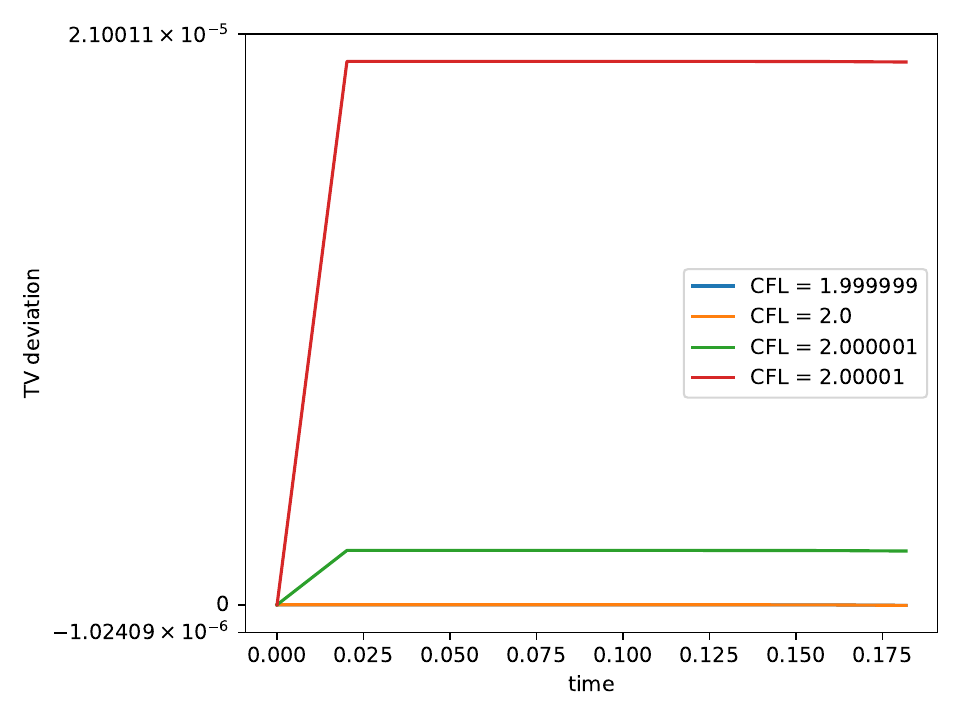}
    \end{minipage} \hfill
    \begin{minipage}{0.49\textwidth}
    \centering
        \begin{tabular}{c|c}
        CFL & $\max_n \text{TV}(\mathbf U(t^n))-\text{TV}(\mathbf U(t^{n-1}))$ \\\hline
        1.999999 & -3.55e-15\\
        2 & -3.99e-15\\
        2.000001 & 1.99e-06\\
        2.00001 & 1.99e-05
    \end{tabular}
    \end{minipage}
    \caption{Deviation of the total variation from the initial total variation (left) and maximum variation of TV (right) for different CFL numbers on a Burgers' double Riemann Problem test with MPE and upwind}
    \label{fig:TV_deviation}
\end{figure}
We test for different explicit CFL numbers close to the critical value 2 to check if the found condition is valid and sharp, i.e., $\frac{\Delta t^n}{\Delta x} \max_i{U_i^n}\leq \text{CFL}$. In figure~\ref{fig:TV_deviation}, we plot the difference between the TV at a given time and the TV at the initial time. We clearly see that for $\text{CFL}\leq 2$ the simulations are TVD while as soon as we cross the value 2, the simulations are not TVD.
Moreover, the maximum of value of the difference between the TV at a time step and at the previous time step is reported in the right of Figure~\ref{fig:TV_deviation} and confirms once again that even imposing the CFL at an explicit level, we sharply obtain the TVD property for $\text{CFL}\leq 2$.

Numerically, we have observed that other high order schemes do not share the same property. For example, we could not find any CFL for which MPDeC of order 3 and 4 would be TVD, while for MPDeC of order 2 we obtained numerically the same CFL bound as MPE. \\

During the preparation of this work the authors used ChatGPT (OpenAI) in order to refer to the compactness of functions of bounded variation (Lemma \ref{lem:delta}). After using this tool/service, the author(s) reviewed and edited the content as needed and take full responsibility for the content of the published article.

\bibliographystyle{siam}
\bibliography{literature}

\begin{thebibliography}{10}

\bibitem{abgrall2023}
{\sc R.~Abgrall}, {\em A personal discussion on conservation, and how to
  formulate it}, in Finite volumes for complex applications X -- Volume 1.
  Elliptic and parabolic problems. FVCA 10, Strasbourg, France, October 30 --
  November 3, 2023. Invited contributions, Cham: Springer, 2023, pp.~3--19.

\bibitem{ambrosio2000functions}
{\sc L.~Ambrosio, N.~Fusco, and D.~Pallara}, {\em Functions of bounded
  variation and free discontinuity problems}, Oxford university press, 2000.

\bibitem{ABBKP2004}
{\sc E.~Audusse, F.~Bouchut, M.-O. Bristeau, R.~Klein, and B.~Perthame}, {\em A
  fast and stable well-balanced scheme with hydrostatic reconstruction for
  shallow water flows}, SIAM J. Sci. Comput., 25 (2004), pp.~2050--2065.

\bibitem{avila_extension_2021}
{\sc A.~I. \'Avila, G.~J. Gonz\'alez, S.~Kopecz, and A.~Meister}, {\em
  Extension of modified {Patankar}–{Runge}–{Kutta} schemes to nonautonomous
  production–destruction systems based on {Oliver}’s approach}, Journal of
  Computational and Applied Mathematics, 389 (2021), p.~113350.

\bibitem{AKM2020}
{\sc A.~I. \'Avila, S.~Kopecz, and A.~Meister}, {\em A comprehensive theory on
  generalized {BBKS} schemes}, Appl. Numer. Math., 157 (2020), pp.~19--37.

\bibitem{GRP_2024}
{\sc M.~Ben-Artzi}, {\em {GRP} -- a direct {Godunov} extension}, J. Comput.
  Phys., 519 (2024), p.~23.
\newblock Id/No 113388.

\bibitem{berthon2025towards}
{\sc C.~Berthon, V.~Michel-Dansac, and A.~Thomann}, {\em Towards a fully
  well-balanced and entropy-stable scheme for the euler equations with gravity:
  preserving isentropic steady solutions}, Mathematics of Computation,  (2025).

\bibitem{birken_conservation_2022}
{\sc P.~Birken and V.~Linders}, {\em Conservation {Properties} of {Iterative}
  {Methods} for {Implicit} {Discretizations} of {Conservation} {Laws}}, J Sci
  Comput, 92 (2022), p.~60.

\bibitem{burchard_high-order_2003}
{\sc H.~Burchard, E.~Deleersnijder, and A.~Meister}, {\em A high-order
  conservative {Patankar}-type discretisation for stiff systems of
  production–destruction equations}, Appl. Numer. Math., 47 (2003),
  pp.~1--30.

\bibitem{CF2013}
{\sc M.~H. Carpenter and T.~C. Fisher}, {\em High-Order Entropy Stable
  Formulations for Computational Fluid Dynamics}.

\bibitem{ciallella2025high}
{\sc M.~Ciallella, L.~Micalizzi, V.~Michel-Dansac, P.~{\"O}ffner, and
  D.~Torlo}, {\em A high-order, fully well-balanced, unconditionally
  positivity-preserving finite volume framework for flood simulations},
  GEM-International Journal on Geomathematics, 16 (2025), pp.~1--33.

\bibitem{ciallella2022arbitrary}
{\sc M.~Ciallella, L.~Micalizzi, P.~{\"O}ffner, and D.~Torlo}, {\em An
  arbitrary high order and positivity preserving method for the shallow water
  equations}, Computers \& Fluids, 247 (2022), p.~105630.

\bibitem{dutt_spectral_2000}
{\sc A.~Dutt, L.~Greengard, and V.~Rokhlin}, {\em Spectral {Deferred}
  {Correction} {Methods} for {Ordinary} {Differential} {Equations}}, BIT
  Numerical Mathematics, 40 (2000), pp.~241--266.

\bibitem{eymard2022finite}
{\sc R.~Eymard, T.~Gallou{\"e}t, R.~Herbin, and J.-C. Latch{\'e}}, {\em Finite
  volume schemes and lax--wendroff consistency}, Comptes Rendus. M{\'e}canique,
  350 (2022), pp.~1--13.

\bibitem{fisher20213}
{\sc T.~C. Fisher, M.~H. Carpenter, J.~Nordstr{\"o}m, N.~K. Yamaleev, and
  C.~Swanson}, {\em Discretely conservative finite-difference formulations for
  nonlinear conservation laws in split form: theory and boundary conditions},
  J. Comput. Phys., 234 (2013), pp.~353--375.

\bibitem{Weak_FV_2019}
{\sc T.~Gallou{\"e}t, R.~Herbin, and J.-C. Latch{\'e}}, {\em On the weak
  consistency of finite volumes schemes for conservation laws on general
  meshes}, S\(\vec{\text{e}}\)MA J., 76 (2019), pp.~581--594.

\bibitem{godlewski_numerical_1996}
{\sc E.~Godlewski and P.-A. Raviart}, {\em Numerical {Approximation} of
  {Hyperbolic} {Systems} of {Conservation} {Laws}}, vol.~118 of Applied
  {Mathematical} {Sciences}, Springer New York, New York, NY, 1996.

\bibitem{huang_positivity-preserving_2019}
{\sc J.~Huang and C.-W. Shu}, {\em Positivity-{Preserving} {Time}
  {Discretizations} for {Production}–{Destruction} {Equations} with
  {Applications} to {Non}-equilibrium {Flows}}, Journal of Scientific
  Computing, 78 (2019), pp.~1811--1839.

\bibitem{huang2019third}
{\sc J.~Huang, W.~Zhao, and C.-W. Shu}, {\em A third-order unconditionally
  positivity-preserving scheme for production--destruction equations with
  applications to non-equilibrium flows}, Journal of Scientific Computing, 79
  (2019), pp.~1015--1056.

\bibitem{IzginThesis}
{\sc T.~Izgin}, {\em A Unifying Theory for Runge-Kutta-like Time Integrators:
  Convergence and Stability}, PhD thesis, University of Kassel, 2024.

\bibitem{NSARK}
{\sc T.~Izgin, D.~I. Ketcheson, and A.~Meister}, {\em Order conditions for
  {R}unge--{K}utta-like methods with solution-dependent coefficients}, Commun.
  Appl. Math. Comput. Sci., 20 (2025), pp.~29--66.

\bibitem{IMPV2025}
{\sc G.~Izzo, E.~Messina, M.~Pezzella, and A.~Vecchio}, {\em Modified patankar
  linear multistep methods for production-destruction systems}, Journal of
  Scientific Computing, 102 (2025), p.~87.

\bibitem{kopecz_order_2018}
{\sc S.~Kopecz and A.~Meister}, {\em On order conditions for modified
  {Patankar}-{Runge}-{Kutta} schemes}, Appl. Numer. Math., 123 (2018),
  pp.~159--179.

\bibitem{KM18Order3}
\leavevmode\vrule height 2pt depth -1.6pt width 23pt, {\em Unconditionally
  positive and conservative third order modified {P}atankar-{R}unge-{K}utta
  discretizations of production-destruction systems}, BIT, 58 (2018),
  pp.~691--728.

\bibitem{kroner_numerical_1997}
{\sc D.~Kröner}, {\em Numerical schemes for conservation laws},
  Wiley-{Teubner} series, advances in numerical mathematics, Wiley; Teubner,
  Chichester; New York; Stuttgart, 1997.

\bibitem{zbMATH07745801}
{\sc D.~Kuzmin and H.~Hajduk}, {\em Property-preserving numerical schemes for
  conservation laws}, Singapore: World Scientific, 2023.

\bibitem{kuzmin2022limiter}
{\sc D.~Kuzmin, H.~Hajduk, and A.~Rupp}, {\em Limiter-based entropy
  stabilization of semi-discrete and fully discrete schemes for nonlinear
  hyperbolic problems}, Comput. Methods Appl. Mech. Eng., 389 (2022), p.~28.
\newblock Id/No 114428.

\bibitem{arXiv:2308.14872}
{\sc D.~Kuzmin, M.~Luk{\'a}cova-Medvid'ov{\'a}, and P.~{\"O}ffner}, {\em
  Consistency and convergence of flux-corrected finite element methods for
  nonlinear hyperbolic problems}.
\newblock Preprint, {arXiv}:2308.14872 [math.{NA}] (2023), 2023.

\bibitem{lax_systems_1960}
{\sc P.~Lax and B.~Wendroff}, {\em Systems of conservation laws}, Comm. Pure
  Appl. Math., 13 (1960), pp.~217--237.

\bibitem{leveque_numerical_1992}
{\sc R.~J. LeVeque}, {\em Numerical {Methods} for {Conservation} {Laws}},
  Birkhäuser Basel, Basel, 1992.

\bibitem{LB23_Krylov_consistency}
{\sc V.~Linders and P.~Birken}, {\em On the consistency of arnoldi-based krylov
  methods for conservation laws}, PAMM, 23 (2023), p.~e202200157.

\bibitem{LB24}
\leavevmode\vrule height 2pt depth -1.6pt width 23pt, {\em Locally conservative
  and flux consistent iterative methods}, SIAM Journal on Scientific Computing,
  46 (2024), pp.~S424--S444.

\bibitem{MCD2020}
{\sc A.~Martiradonna, G.~Colonna, and F.~Diele}, {\em {\it {G}e{C}o}:
  {G}eometric {C}onservative nonstandard schemes for biochemical systems},
  Appl. Numer. Math., 155 (2020), pp.~38--57.

\bibitem{MO2014}
{\sc A.~Meister and S.~Ortleb}, {\em On unconditionally positive implicit time
  integration for the {DG} scheme applied to shallow water flows}, Internat. J.
  Numer. Methods Fluids, 76 (2014), pp.~69--94.

\bibitem{offner2020arbitrary}
{\sc P.~{\"O}ffner and D.~Torlo}, {\em Arbitrary high-order, conservative and
  positivity preserving {Patankar}-type deferred correction schemes}, Applied
  Numerical Mathematics, 153 (2020), pp.~15--34.

\bibitem{Ortleb2014}
{\sc S.~Ortleb}, {\em Positivity preserving implicit and partially implicit
  time integration methods in the context of the {DG} scheme applied to shallow
  water flows}, in Finite volumes for complex applications {VII}. {M}ethods and
  theoretical aspects, vol.~77 of Springer Proc. Math. Stat., Springer, Cham,
  2014, pp.~431--438.

\bibitem{OrtlebH2017}
{\sc S.~Ortleb and W.~Hundsdorfer}, {\em {Patankar-type Runge-Kutta schemes for
  linear PDEs}}, vol.~1863, 07 2017, p.~320008.

\bibitem{perthame1996positivity}
{\sc B.~Perthame and C.-W. Shu}, {\em On positivity preserving finite volume
  schemes for {Euler} equations}, Numerische Mathematik, 73 (1996),
  pp.~119--130.

\bibitem{RB2009}
{\sc M.~Ricchiuto and A.~Bollermann}, {\em Stabilized residual distribution for
  shallow water simulations}, J. Comput. Phys., 228 (2009), pp.~1071--1115.

\bibitem{shi_2018_local}
{\sc C.~Shi and C.-W. Shu}, {\em On local conservation of numerical methods for
  conservation laws}, Comput. Fluids, 169 (2018), pp.~3--9.

\bibitem{SS2018}
{\sc C.~Shi and C.-W. Shu}, {\em On local conservation of numerical methods for
  conservation laws}, Computers \& Fluids, 169 (2018), pp.~3--9.
\newblock Recent progress in nonlinear numerical methods for time-dependent
  flow \& transport problems.

\bibitem{shu_efficient_1988}
{\sc C.-W. Shu and S.~Osher}, {\em Efficient implementation of essentially
  non-oscillatory shock-capturing schemes}, Journal of Computational Physics,
  77 (1988), pp.~439--471.

\bibitem{Tadmor2003}
{\sc E.~Tadmor}, {\em Entropy stability theory for difference approximations of
  nonlinear conservation laws and related time-dependent problems}, Acta
  Numer., 12 (2003), pp.~451--512.

\bibitem{TOR2022}
{\sc D.~Torlo, P.~\"Offner, and H.~Ranocha}, {\em Issues with
  positivity-preserving {P}atankar-type schemes}, Appl. Numer. Math., 182
  (2022), pp.~117--147.

\bibitem{toth2023total}
{\sc G.~Toth}, {\em Total of time variation diminishing principle for
  conservation laws}, J. Comput. Phys., 494 (2023), pp.~Paper No. 112534, 14.

\bibitem{zhu2024bound}
{\sc F.~Zhu, J.~Huang, and Y.~Yang}, {\em Bound-preserving discontinuous
  {Galerkin methods with modified Patankar time integrations} for chemical
  reacting flows}, Communications on Applied Mathematics and Computation, 6
  (2024), pp.~190--217.

\end{thebibliography}

\end{document}